\documentclass[a4paper, 10pt]{amsart}
\usepackage{amsmath}
\usepackage{amssymb, color, hyperref}

\theoremstyle{plain}
\newtheorem{theorem}{\bf Theorem}[section]
\newtheorem{proposition}[theorem]{\bf Proposition}
\newtheorem{lemma}[theorem]{\bf Lemma}

\theoremstyle{definition}
\newtheorem{example}[theorem]{\bf Example}

\newtheorem{definition}[theorem]{\bf Definition}
\newtheorem{remark}[theorem]{\bf Remark}

\newcommand{\N}{\mathbb N}
\newcommand{\Z}{\mathbb Z}
\newcommand{\R}{\mathbb R}
\newcommand{\Q}{\mathbb Q}

\newcommand{\BF}{\text{\rm BF}}

 \DeclareMathOperator{\ord}{ord}
 
 \DeclareMathOperator{\Ca}{\mathsf {Ca}}
\DeclareMathOperator{\spec}{spec} \DeclareMathOperator{\supp}{supp}
\DeclareMathOperator{\Pic}{Pic} 
 
\DeclareMathOperator{\card}{card} \DeclareMathOperator{\Int}{Int}
\DeclareMathOperator{\End}{End}

\renewcommand{\time}{\negthinspace \times \negthinspace}

\newcommand{\DP}{\negthinspace : \negthinspace}
\newcommand{\red}{{\text{\rm red}}}
\newcommand{\fin}{{\text{\rm fin}}}
\renewcommand{\t}{\, | \,}

\newcommand{\canc}{\text{\rm canc}}
\newcommand{\LK}{\,[\![}
\newcommand{\RK}{]\!]}

\numberwithin{equation}{section}


\begin{document}

\title{Factorization Theory in Commutative Monoids}

\author{Alfred Geroldinger  and Qinghai Zhong}

\address{University of Graz, NAWI Graz \\
Institute for Mathematics and Scientific Computing \\
Heinrichstra{\ss}e 36\\
8010 Graz, Austria}
\email{alfred.geroldinger@uni-graz.at, qinghai.zhong@uni-graz.at}
\urladdr{https://imsc.uni-graz.at/geroldinger, https://imsc.uni-graz.at/zhong}


\begin{abstract}
This is a survey on factorization theory. We discuss finitely generated monoids (including affine monoids), primary monoids (including numerical monoids), power sets with set addition, Krull monoids and their various generalizations, and the multiplicative monoids of domains (including Krull domains, rings of integer-valued polynomials, orders in algebraic number fields) and of their ideals. We offer  examples for all these classes of monoids and discuss their main arithmetical finiteness properties. These describe the structure of their sets of lengths, of the unions of sets of lengths, and their catenary degrees. We also provide examples where these finiteness properties do not hold.
\end{abstract}

\maketitle


\section{Introduction} \label{1}

Factorization theory emerged from algebraic number theory. The ring of integers of an algebraic number field is factorial if and only if it has class number one, and the class group was always considered as a measure for the non-uniqueness of factorizations. Factorization theory has  turned this idea into concrete results.
In 1960 Carlitz proved (and this is  a starting point of the area) that the ring of integers is half-factorial (i.e., all sets of lengths are singletons) if and only if the class number is at most two. In the 1960s Narkiewicz started a systematic study of counting functions associated with arithmetical properties in rings of integers. Starting in the  late 1980s,  theoretical properties of factorizations were studied in  commutative semigroups and in  commutative integral domains, with a focus on Noetherian  and Krull domains (see \cite{Ge88, Ge-Le90, HK90b}; \cite{An-An-Za90} is the first in a series of papers by Anderson, Anderson, Zafrullah, and  \cite{An97} is a conference volume from the 1990s).

From these beginnings factorization theory branched out, step by step,  into various subfields of algebra including commutative and non-commutative ring theory, module theory, and abstract semigroup theory and today  is considered as a structure theory of the arithmetic of a broad variety of objects. In this survey, we discuss finitely generated monoids (including affine monoids), Krull monoids (including Krull and Dedekind domains), power monoids (including the set of finite nonempty subsets of the nonnegative integers with set addition as its operation), strongly primary monoids (including numerical monoids and local one-dimensional Noetherian domains), and weakly Krull monoids (including orders in algebraic number fields). The main aim   of factorization theory is to
describe the various phenomena of non-uniqueness of factorizations by arithmetical invariants and to study the interdependence of these invariants and the classical algebraic invariants of the underlying algebraic structures. We discuss three long-term goals (\hyperlink{A}{Problem A}, \hyperlink{B}{Problem B}, and \hyperlink{C}{Problem C}) of this area.

It  turns out that abstract semigroup theory provides a most suitable frame for the formulation of arithmetic concepts, even for studying factorizations in domains. A reason for this lies in the use of one of its  main conceptual tools,  transfer homomorphisms.  Objects of interest $H$ are oftentimes studied via simpler objects $B$ and associated transfer homomorphisms $\theta \colon H \to B$, which allow one  to pull back arithmetical properties from $B$ to $H$ (see Definition \ref{4.4} and Proposition \ref{4.5}).

In Section \ref{2} we present semigroups from ring theory (semigroups of ideals and of modules) and power monoids (stemming from additive combinatorics), and we introduce the arithmetical concepts discussed later in the paper (including sets of lengths and their unions, sets of distances, and catenary degrees). Theorem \ref{3.1} in Section \ref{3} gathers  the main arithmetical finiteness results for finitely generated monoids. In the next sections, we present Krull monoids, transfer Krull monoids, and weakly Krull monoids. We offer examples of such monoids, discuss their arithmetical properties, show how some of them can be pulled back from finitely generated monoids (Theorem \ref{5.5}), and show that none of these arithmetical finiteness properties need to hold in general (Remark \ref{5.7}).

Various aspects of factorization theory could not be covered in this survey. These include factorizations in non-commutative rings and semigroups (\cite{Sm16a}), factorizations in commutative rings with zero-divisors (\cite{An-Ch11a}), the arithmetic of non-atomic,  non-BF, and non-Mori domains (\cite{An-Ma-Va12,Co-Go19a,Co-Ha18a}), and factorizations into distinguished elements that are not irreducible (e.g., factorizations into radical ideals and others  \cite{Fo-Ho-Lu13a, Sa-Za14, Ol-Re19a, Ol-Re20a,Ch-Re20a}).

\section{Background on  monoids and their arithmetic} \label{2}

We denote by $\N$ the set of positive integers. For rational numbers $a, b \in \Q$, $[a, b ] = \{ x \in \Z \colon a \le x \le b\}$ means the discrete interval between $a$ and $b$. For subsets $A, B \subset \Z$, $A+B = \{a+b \colon a \in A, b \in B \}$ denotes their sumset and, for every $k \in \N$, $kA = A + \ldots + A$ is the $k$-fold sumset of $A$. The set of distances $\Delta (A)$ is the set of all $d \in \N$ for which there is $a \in A$ such that $A \cap [a, a+d] = \{a, a+d\}$. If $A \subset \N_0$, then $\rho (A) = \sup (A \cap \N)/\min (A \cap \N) \in \Q_{\ge 1} \cup \{\infty\}$ denotes the elasticity of $A$ with the convention that  $\rho ( A)=1$ if $A \cap \N = \emptyset$. If $d \in \N$ and $M \in \N_0$, then a subset $L \subset \Z$ is called an {\it almost arithmetical progression} (AAP) with difference $d$ and bound $M$ if
\[
L = y + (L' \cup L^* \cup L'') \subset y + d \Z \,,
\]
where $y \in \Z$ is a shift parameter, $L^*$ is a nonempty arithmetical progression with difference $d$ such that $\min L^* = 0$, $L' \subset [-M,-1]$, and $L'' \subset \sup L^* + [1,M]$ (with the convention that $L''=\emptyset$ if $L^*$ is infinite).

\subsection{\bf Monoids.} \label{Monoids}
Let $H$ be a multiplicatively written commutative semigroup. We denote by $H^{\times}$ the group of invertible elements of $H$. We say that $H$ is reduced if $H^{\times}=\{1\}$ and we denote by $H_{\red}=\{aH^{\times}\mid a\in H\}$ the associated reduced semigroup of $H$.  The semigroup $H$ is said to be
\begin{itemize}
\item {\it cancellative} if $a,b,u \in H$ and $au=bu$  implies that $a=b$;
\item {\it unit-cancellative} if $a,u\in H$ and $a=au$ implies that $u\in H^{\times}$.
\end{itemize}
By definition, every cancellative semigroup is unit-cancellative. If $H$ is a unit-cancellative semigroup, then we define, for two elements $a, b \in H$, that $a \sim b$ if there is $c \in H$ such that $ac=bc$. This is a congruence relation on $H$ and the monoid $H_{\canc} = H/\!\!\sim$ is the associated cancellative monoid of $H$.
If $H$ is cancellative, then $\mathsf q (H)$ denotes the quotient group of $H$,
\begin{itemize}
\item $\widehat H = \{ x \in \mathsf q (H) \colon \text{there is $c \in H$ such that $cx^n \in H $ for all $n \in \N$} \}$
       is the {\it complete integral closure} of $H$, and
\item $\widetilde H = \{ x \in \mathsf q (H) \colon x^n \in H \ \text{for some} \ n \in \N \}$ is the {\it root closure} (also called the {\it normalization}) of $H$.
\end{itemize}
We say that $H$ is {\it completely integrally closed} if $H = \widehat H$ and that it is {\it root closed} (or {\it normal}) if $H = \widetilde H$. For a set $P$, let $\mathcal F (P)$ denote the free abelian monoid with basis $P$. Every $a \in \mathcal F (P)$ has a unique representation in the form
\[
a = \prod_{p \in P} p^{\mathsf v_p (a)} \,,
\]
where $\mathsf v_p \colon H \to \N_0$ is the $p$-adic valuation of $a$. We call $|a| = \sum_{p \in P} \mathsf v_p (a) \in \N_0$ the {\it length} of $a$ and $\supp (a) = \{ p \in P \colon \mathsf v_p (a) > 0 \} \subset P$ the {\it support} of $a$.

\medskip
\centerline{\it Throughout this paper, a monoid means a }
\centerline{\it commutative unit-cancellative semigroup with identity element.}
\medskip

Let $H$ be a monoid. For two elements $a, b \in H$ we say that $a$ divides $b$ (we write $a \t b$) if $b \in aH$ and if $aH = bH$ (equivalently, $aH^{\times} = b H^{\times}$), then $a$ and $b$ are called {\it associated} (we write $a \simeq b$). The element $a$ is called {\it irreducible} (or an {\it atom}) if $a = bc$ with $b,c \in H$ implies that $b \in H^{\times}$ or $c \in H^{\times}$. We denote by $\mathcal A (H)$ the set of atoms of $H$. A submonoid $S \subset H$ is said to be {\it divisor-closed} if every divisor $a \in H$ of some element $b \in S$ lies in $S$. A monoid homomorphism $\varphi \colon H \to D$ is a {\it divisor homomorphism} if $a, b \in H$ and $\varphi (a) \t \varphi (b)$ (in $S$) implies that $a \t b$ (in $H$). If the inclusion $H \hookrightarrow D$ is a divisor homomorphism, then $H$ is called a {\it saturated} submonoid of $D$. For a subset $E \subset H$ we denote by
\begin{itemize}
\item $[E] \subset H$ the smallest submonoid of $H$ containing $E$, and by
\item $\LK E \RK \subset H$ the smallest divisor-closed submonoid of $H$ containing $E$.
\end{itemize}
Clearly, $\LK E \RK$ is the set of all $a \in H$ dividing some element $b \in [E]$. If $E = \{a_1, \ldots, a_m\}$, then we write $[a_1, \ldots, a_m] = [E]$ and $\LK a_1, \ldots, a_m \RK = \LK E \RK $.

A subset $\mathfrak a \subset H$ is an $s$-ideal if $\mathfrak a H = \mathfrak a$ and $H$ is {\it $s$-Noetherian} if $H$ satisfies the ACC (ascending chain condition) on $s$-ideals. We denote by $s$-$\spec (H)$ the set of prime $s$-ideals and by $\mathfrak X (H) \subset s$-$\spec (H)$ the set of minimal nonempty prime $s$-ideals. Suppose that $H$ is cancellative. For subsets $A, B \subset H$, $(A \DP B ) = \{ c \in \mathsf q (H) \colon cB \subset A \}$ and $A_v = \big( H \colon (H \DP A) \big)$ is the {\it $v$-ideal}  (or {\it divisorial ideal}) generated by $A$. If $A_v = A$, then $A$ is a divisorial ideal. The monoid $H$ is a {\it Mori monoid} (or  {\it $v$-Noetherian}) if it satisfies the ACC for divisorial ideals. If $A, B \subset H$ are divisorial ideals, then $A \cdot_v B = (AB)_v$ is the $v$-product of $A$ and $B$. We denote by $\mathcal I_v (H)$ the semigroup of divisorial ideals of $H$ (equipped with $v$-multiplication) and by $\mathcal I_v^* (H)$ the subsemigroup of $v$-invertible divisorial ideals. Clearly, $\mathcal I_v^* (H)$ is cancellative and if $H$ is a Mori monoid, then $\mathcal I_v (H)$ is a monoid and $\mathcal I_v^* (H)$ is a Mori monoid. For any undefined concepts in ideal theory we refer to \cite{HK98}.

The monoid $H$ is said to be {\it finitely generated} if there is a finite set $E \subset H$ such that $H = [E]$. Every finitely generated monoid is $s$-Noetherian and the converse holds if $H$ is cancellative (\cite[Theorem 3.6]{HK98}).
A monoid is called {\it affine} if it is finitely generated and isomorphic to a submonoid of a finitely generated free abelian group (equivalently, a commutative semigroup is  affine    if it is reduced, cancellative, finitely generated, and its quotient group is torsion-free).

If not stated otherwise, then a ring means a commutative ring with identity element. Let $R$ be a ring. Then $R^{\bullet}$ denotes the semigroup of regular elements, and $R^{\bullet}$ is a cancellative monoid. Rings with the property that $au=a$ implies that $u \in R^{\times}$ or $a=0$ are called pr{\'e}simplifiable  in \cite{An-Al17a}.
Ring theory gives rise to the following classes of monoids that are of central interest in factorization theory.

\begin{example}[\bf Monoids from ring theory] \label{2.1}~

1. (Semigroups of ideals) Let $R$ be a commutative integral domain. We denote by $\overline R$ its integral closure and by $\widehat R$ its complete integral closure. Further, let $\mathcal H (R)$ be the semigroup of nonzero principal ideals,  $\mathcal I^* (R)$ be the semigroup of invertible ideals,   $\mathcal I (R)$ be the semigroup of all nonzero ideals,  and $\mathcal F (R)$ be the semigroup of nonzero fractional ideals, all equipped with usual ideal multiplication. Then $\mathcal F (R)^{\times}$, the group of units of $\mathcal F (R)$, is the group of invertible fractional ideals and this is the quotient group of $\mathcal I^* (R)$.  Furthermore,  $\mathcal H (R) \cong (R^{\bullet})_{\red}$, the inclusion $\mathcal H (R) \hookrightarrow \mathcal I^* (R)$ is a cofinal divisor homomorphism, $\mathcal I^* (R) \subset \mathcal I (R)$ is a divisor-closed submonoid, the prime elements of $\mathcal I^* (R)$ are precisely the invertible prime ideals, and $\Pic (R) = \mathcal F (R)^{\times}/\mathsf q (\mathcal H (R))$ is the {\it Picard group} of $R$. Suppose that $R$ is Noetherian. If $I, J \in \mathcal I (R)$ with $IJ = I$, then $IJ^n = I$ whence $\{0\} \ne I \subset \cap_{n \ge 0}J^n$. Since $R$ satisfies Krull's Intersection Theorem, it follows that $J=R$. Thus $\mathcal I (R)$ is unit-cancellative whence a monoid in the present sense.

The above constructions generalize to monoids of $r$-ideals for general ideal systems $r$ (the interested reader may want to consult \cite{HK98, HK11b,Re12a}). In the present paper we restrict ourselves to usual ring ideals, to usual semigroup ideals ($s$-ideals), and to divisorial ideals of monoids and domains. We use that the semigroup $\mathcal I_v (R)$  of divisorial ideals of $R$ (respectively the monoid $\mathcal I_v^* (R)$ of $v$-invertible divisorial ideals of $R$) are isomorphic to the semigroup  $\mathcal I_v (R^{\bullet})$  of divisorial ideals of $R^{\bullet}$ (respectively to the monoid $\mathcal I_v^* (R^{\bullet})$ of $v$-invertible divisorial ideals of $R^{\bullet}$). In particular,  $R$ is a Mori domain if and only if its monoid $R^{\bullet}$ is a Mori monoid.

Atomic domains $R$ having only finitely many non-associated atoms are called {\it Cohen-Kaplansky domains} and they are characterized by each of the following equivalent properties (\cite[Theorem 4.3]{An-Mo92}){\rm \,:}
\begin{itemize}
\item[(a)] $\mathcal H (R)$ is finitely generated.
\item[(b)] $\mathcal I (R)$ is finitely generated.
\item[(c)] $\mathcal I^* (R)$ is finitely generated.
\item[(d)] $\overline R$ is a semilocal principal ideal domain, $\overline R/(R \DP \overline R)$ is finite, and $|\max (R)|= |\max (\overline R)|$.
\end{itemize}

2. (Semigroups of modules) Let $R$ be a not necessarily commutative ring and let $\mathcal C$ be a class of right $R$-modules that is closed under finite direct sums, direct summands, and isomorphisms. This means, whenever $M, M_1, M_2$ are $R$-modules with $M \cong M_1 \oplus M_2$, then $M$ lies in $\mathcal C$ if and only if $M_1$ and $M_2$ lies in $\mathcal C$. Let $\mathcal V (\mathcal C)$ be the set of isomorphism classes of $\mathcal C$ (we tacitly assume that this is indeed a set) where the operation is induced by forming direct sums. Then $(\mathcal V (\mathcal C), +)$ is a reduced commutative semigroup. By a result of Bergman-Dicks (\cite[Theorems 6.2 and 6.4]{Be74a} and \cite[page 315]{Be-Di78}), every commutative reduced semigroup $S$ with $S = \LK a \RK$ for some $a \in S$ is isomorphic to a semigroup of modules (indeed, one may take the class of finitely generated projective right $R$-modules over a hereditary $k$-algebra).

If each module $M$ in $\mathcal C$ is Noetherian (or artinian), then it is a finite direct sum of indecomposable modules and hence $\mathcal V (\mathcal C)$ is atomic. If the endomorphism rings $\End_R (M)$ are local for all indecomposable  modules in $\mathcal C$, then direct sum decomposition is unique whence $\mathcal V ( \mathcal C)$ is free abelian (in other words, the Krull-Remak-Schmidt-Azumaya Theorem holds). A module is said to be directly finite (or Dedekind finite) if it is not isomorphic to a proper direct summand of itself (\cite{Go79a, La99a}). Thus the semigroup $\mathcal V (\mathcal C)$ is unit-cancellative (whence a monoid in the present sense) if and only if all modules in $\mathcal C$ are directly finite.
The idea, to look at direct-sum decomposition of modules, from the viewpoint of factorization theory was pushed forward by Facchini, Wiegand, et al (for a survey see \cite{Ba-Wi13a}). We meet semigroups of modules again in  Example \ref{4.2}.4.
\end{example}

We end this subsection with a class of monoids stemming from additive combinatorics.

\begin{example}[\bf Power monoids] \label{2.2}~
Let $H$ be an {\it additively written} torsion-free monoid. The {\it power monoid} $\mathcal P_{\fin} (H)$ of $H$ is the set of all finite nonempty subsets of $H$, endowed with set addition as operation (thus, if $A, B \in \mathcal P_{\fin} (H)$, then $A+B = \{a+b \colon a \in A, b \in B \}$ is the sumset of $A$ and $B$). Clearly, $\mathcal P_{\fin} (H)$ is a commutative semigroup and if $0_H \in H$ is the identity element of $H$, then $\{0_H\}$ is the identity element of $\mathcal P_{\fin} (H)$. The subset $\mathcal P_{\fin, \times} (H) \subset \mathcal P_{\fin} (H)$, which consists of those finite nonempty subsets  $A \subset H$ with $A \cap H^{\times} \ne \emptyset$, is a divisor-closed submonoid of $\mathcal P_{\fin} (H)$, called the {\it restricted power monoid} of $H$. Power monoids of monoids were introduced by Tringali et al. and studied in an abstract framework  (\cite{Fa-Tr18a, An-Tr20a, Tr20}). For simplicity of presentation, we  restrict ourselves to $\mathcal P_{\fin} (\N_0)$ and to   $\mathcal P_{\fin, 0} (\N_0)$ consisting of all finite nonempty subsets of $\N_0$ containing $0$. Finite nonempty subsets of the (nonnegative) integers and their sumsets are the primary objects of study in arithmetic combinatorics (\cite{Ta-Vu06,Ge-Ru09,Gr13a}).
\end{example}

\subsection{\bf Arithmetical concepts.} \label{arithmetical}
Let $H$ be a monoid. The free abelian monoid $\mathsf Z (H) = \mathcal F ( \mathcal A (H_{\red}))$ is the {\it factorization monoid} of $H$ and $\pi \colon \mathsf Z (H) \to H_{\red}$, defined by $\pi (u) = u$ for all $u \in \mathcal A (H_{\red})$, is the {\it factorization homomorphism} of $H$. For an element $a \in H$,
\begin{itemize}
\item $\mathsf Z_H (a) = \mathsf Z (a) = \pi^{-1} (aH^{\times}) \subset \mathsf Z (H)$ is the {\it set of factorizations} of $a$, and
\item $\mathsf L_H (a) = \mathsf L (a) = \{ |z| \colon z \in \mathsf Z (a) \} \subset \N_0$ is the {\it set of lengths} of $a$.
\end{itemize}
Thus, $\mathsf L (a) = \{0\}$ if and only if $a \in H^{\times}$ and $\mathsf L (a) = \{1\}$ if and only if $a \in \mathcal A (H)$. Then
\[
\mathcal L (H) = \{ \mathsf L (a) \colon a \in H \}
\]
is the {\it system of sets of lengths} of $H$. The monoid $H$ is said to be
\begin{itemize}
\item {\it atomic} if $\mathsf Z (a) \ne \emptyset$ for all $a \in H$ (equivalently, every non-invertible element of $H$ can be written as a finite product of atoms of $H$);

\item {\it factorial} if $|\mathsf Z (a)|=1$ for all $a \in H$;

\item {\it half-factorial} if $|\mathsf L (a)|=1$ for all $a \in H$;

\item a BF-{\it monoid} (bounded factorization monoid) if $\mathsf L (a)$ is finite and nonempty for all $a \in H$.
\end{itemize}
A monoid $H$ is factorial if and only if $H_{\red}$ is free abelian.
Every Mori monoid is a BF-monoid, every BF-monoid satisfies the ACC on principal ideals, and every monoid satisfying the ACC on principal ideals is atomic. The main focus of factorization theory is  on BF-monoids and this will also be the case in the present paper. For any undefined notion we refer to \cite{Ge-HK06a}.

Suppose that $H$ is a BF-monoid. Then $\mathcal L (H) \subset \mathcal P_{\fin} (\N_0)$ and for any subset $\mathcal L \subset \mathcal P_{\fin} (\N_0)$ we define the following invariants describing the structure of $\mathcal L$. We denote by
\[
\Delta (\mathcal L) = \bigcup_{L \in \mathcal L } \Delta (L) \subset \N \quad \text{the {\it set of distances} of} \ \mathcal L \,,
\]
and by
\[
\mathcal R (\mathcal L) = \{\rho (L) \colon L \in \mathcal L  \} \subset \Q_{\ge 1}  \quad \text{the {\it set of elasticities} of} \ \mathcal L \,.
\]
Then $\rho (\mathcal L) = \sup \mathcal R (\mathcal L)$ is called the {\it elasticity} of $\mathcal L$ and we say that the elasticity is accepted if there is an $L \in \mathcal L$ with $\rho (L)=\rho (\mathcal L)$.    For every $k \in \N_0$,
\[
\mathcal U_k (\mathcal L) = \bigcup_{k \in L \in \mathcal L } L \subset \N \quad \text{is the {\it union of sets of $\mathcal L$} containing $k$} \,,
\]
and $\rho_k (\mathcal L) = \sup \mathcal U_k (\mathcal L)$ is the {\it $k$-th elasticity} of $\mathcal L$.
If $\mathcal L = \mathcal L (H)$, then we briefly set $\Delta (H) = \Delta \bigl( \mathcal L (H) \bigr)$ and similarly for the other invariants. Thus, by definition,  $H$ is half-factorial if and only if $\Delta (H) = \emptyset$ if and only if $\mathcal R (H) = \{1\}$. Furthermore, if $H \ne H^{\times}$, then $\mathcal U_0 (H) = \{0\}$ and $\mathcal U_1 (H) = \{1\}$.

In many settings unions of sets of lengths as well as sets of lengths have a well-defined structure. For their description we need the concept of an AAMP (almost arithmetical multiprogression).  Let $d \in \N$, \ $M \in \N_0$ \ and \ $\{0,d\} \subset \mathcal D
\subset [0,d]$. A subset $L \subset \Z$ is called an  {\rm AAMP}  with \ {\it difference} \ $d$, \ {\it period} \ $\mathcal D$,
      \  and \ {\it bound} \ $M$, \ if
\[
L = y + (L' \cup L^* \cup L'') \, \subset \, y + \mathcal D + d \Z \,, \quad \text{where}
\]
\begin{itemize}
\item  $L^*$ is finite nonempty with $\min L^* = 0$ and $L^* =
       (\mathcal D + d \Z) \cap [0, \max L^*]$, and

\item  $L' \subset [-M, -1]$,  \ $L'' \subset \max L^* + [1,M]$, and $y \in \Z$.
\end{itemize}

Next we define a distance function on the set of factorizations $\mathsf Z (H)$. Two factorizations $z, z' \in \mathsf Z (H)$ can be written in the form
\[
z = u_1 \cdot \ldots \cdot u_{\ell}v_1 \cdot \ldots \cdot v_m \quad \text{and} \quad z' = u_1 \cdot \ldots \cdot u_{\ell}w_1 \cdot \ldots \cdot w_n \,,
\]
where $\ell, m, n \in \N_0$ and $u_1, \ldots, u_{\ell}, v_1, \ldots, v_m, w_1, \ldots, w_n \in \mathcal A (H_{\red})$ are such that $\{v_1, \ldots, v_m\} \cap \{w_1, \ldots, w_n\} = \emptyset$. Then $\mathsf d (z,z') = \max \{m,n\} \in \N_0$ is the {\it distance} between $z$ and $z'$. If $z \ne z'$ with $\pi (z) = \pi (z')$, then
\begin{equation} \label{dist-1}
1 + \big| |z| - |z'| \big| \le \mathsf d (z,z') \quad \text{respectively} \quad 2 + \big| |z| - |z'| \big| \le \mathsf d (z,z')
\end{equation}
if $H$ is cancellative. Let $a \in H$ and $N \in \N_0$. A finite sequence $z_0, \ldots, z_k \in \mathsf Z (a)$ is called an $N$-chain of factorizations if $\mathsf d (z_{i-1}, z_i) \le N$ for all $i \in [1,k]$. Then $\mathsf c_H (a) = \mathsf c (a)$ is the smallest $N \in \N_0 \cup \{\infty\}$ such that any two factorizations $z, z' \in \mathsf Z (a)$ can be concatenated by an $N$-chain. The set
\[
\Ca (H) = \{\mathsf c (a) \colon a \in H \ \text{with} \ \mathsf c (a) > 0 \} \subset \N_0
\]
is the {\it set of (positive) catenary degrees} of $H$ and $\mathsf c (H) = \sup \Ca (H) \in \N_0 \cup \{\infty\}$ is the {\it catenary degree} of $H$. If $H$ is not half-factorial, then the inequalities of \eqref{dist-1} imply that
\begin{equation} \label{dist-2}
1 + \sup \Delta (H) \le \mathsf c (H) \quad \text{resp.} \quad 2 + \sup \Delta (H) \le \mathsf c (H)
\end{equation}
if $H$ is cancellative.

\section{Finitely generated monoids} \label{3}

By Redei's Theorem, every finitely generated commutative semigroup is finitely presented. The idea to describe arithmetical invariants in terms of relations was pushed forward by Chapman and Garc{\'i}a-S{\'a}nchez (\cite{C-G-L-P-R06, C-G-L09} are the first papers in this direction). This point of view laid the foundation for the development of algorithms computing arithmetical invariants in finitely generated monoids (we refer to \cite{GS16a} for a survey, and to \cite{GG-MF-VT15, Ph15a} for a sample of further work in this direction). In particular, for numerical monoids there is a wealth of papers providing algorithms for determining arithmetical invariants and in some cases there are  even  precise values (formulas) for arithmetical invariants (in terms of the atoms or of other algebraic invariants; \cite{MR3493240, GS-ON-We19}). A further class of objects, for which precise formulas for arithmetical invariants are available, will be discussed in Section \ref{6}.

Our first result summarizes the main arithmetical finiteness properties of finitely generated monoids. Its proof is (implicitly) based on Dickson's Lemma stating that a subset of $\N_0^s$ has only finitely minimal points.

\begin{theorem}[\bf Arithmetic of finitely generated monoids] \label{3.1}
Let $H$ be a monoid such that $H_{\red}$ is finitely generated.
\begin{enumerate}
\item The set of  catenary degrees and the set of distances   are finite and $\rho (H) \in \Q$. If $H$ is cancellative, then the elasticity is accepted and there is some $r \in \R_{\ge 1}$ such that $\{ q \in \Q \colon r \le q \le \rho (H) \} \subset \mathcal R (H)$; moreover, $r$ is the only possible limit point of $\{\rho (L) \colon L \in \mathcal L (H) \ \text{with} \ \rho (L) < r\}$.

\item There is $M \in \N_0$ such that, for all $k \in \N$,  the unions $\mathcal U_k (H)$ are finite {\rm AAP}s with difference $\min \Delta (H)$ and bound $M$.

\item If $H$ is cancellative, then there is $M \in \N_0$ such that every $L \in \mathcal L (H)$ is an {\rm AAMP} with difference $d \in \Delta (H)$ and bound $M$.
\end{enumerate}
\end{theorem}

\begin{proof}
1. Suppose that the set of catenary degrees is finite. Then \eqref{dist-2} implies that the set of distances is finite. The elasticity $\rho (H)$ is rational by  \cite[Proposition 3.4]{F-G-K-T17}. Now suppose in addition that $H$ is cancellative. Then the elasticity $\rho (H)$ is accepted by \cite[Theorem 3.1.4]{Ge-HK06a} (this does not hold true in general if $H$ is not cancellative). The claim on the structure of $\mathcal R (H)$ was proved in \cite{Zh19a}.

Now we show   that the catenary degree $\mathsf c (H)$ is finite. We may assume that $H$ is reduced and we denote by  $\pi \colon \mathsf Z (H) \to H$  the factorization homomorphism. We consider the submonoid
\[
S=\{(x,y)\in \mathsf Z (H) \times\mathsf Z (H) \colon \text{ there exists } z\in \mathsf Z (H) \text{ such that }\pi(xz)=\pi(yz)\} \subset \mathsf Z (H) \times \mathsf Z (H)
\]	
and start with the following assertion.
\begin{enumerate}
\item[{\bf A.}\,]  The set
    \begin{align*}
S^*=\{(x,y)\in S\colon &\text{ there exists }z\in \mathsf Z (H) \text{ such that }\pi(xz)=\pi(yz), \text{ but for all } \\
&(x_1,y_1)\in S\setminus\{(1,1),(x,y)\} \text{ with } (x_1,y_1)\mid_S (x,y), \text{ we have } \pi(x_1z)\neq \pi(y_1z)
\}
	\end{align*}
is finite.
\end{enumerate}

{\it Proof of \,{\bf A}}.\, Assume to the contrary that $S^*\subset \mathsf Z (H) \times \mathsf Z (H)$ is infinite. Then there exists a sequence $(x_i,y_i)_{i \ge 1}$ 	with terms from $S^*$ such that $(x_i,y_i)\neq (x_j,y_j)$ and $(x_i,y_i)\mid_{\mathsf Z (H) \times \mathsf Z (H)} (x_j,y_j)$ for distinct $i,j\in \N$ with $i<j$. Since $S\subset \mathsf Z (H) \times \mathsf Z (H)$ is saturated, we have $(x_i,y_i)\mid_S(x_j,y_j)$ for distinct $i,j\in \N$ with $i<j$. For every $i\in \N$, we define
\[
\mathfrak a_i=\{z\in \mathsf Z (H) \colon \text{ there exists} \  (x',y')\in S\setminus\{(1,1)\} \text{ with } (x',y')\mid_S (x_i,y_i)
\text{ such that } \pi(x'z)=\pi(y'z) \}\,.
\]
Then $(\mathfrak a_i)_{i \ge 1}$ is an ascending chain of  $s$-ideals of $\mathsf Z (H)$. Since $\mathsf Z (H)$ is finitely generated, every ascending chain of  $s$-ideals of $\mathsf Z (H)$ is stationary (this proof uses  Dickson's Lemma). Thus there exists $N\in \N$ such that $\mathfrak a_N=\mathfrak a_{N+1}$. Therefore for every $z\in \mathsf Z (H)$ with $\pi(x_{N+1}z)=\pi(y_{N+1}z)$ we have $z \in \mathfrak a_{N+1}=\mathfrak a_N$. By definition of $\mathfrak a_N$,  there is $(x',y')\in S\setminus\{(1,1)\}$ with $(x',y')\mid_S(x_N,y_N)$ such that $\pi(x'z)=\pi(y'z)$, a contradiction to $(x_{N+1},y_{N+1})\in S^*$. \qed[Proof of {\bf A.}]

We assert that
\[
\mathsf c(H)\le M:=\max\{\mathsf d(x,y)\colon (x,y)\in S^* \} \,.
\]
It suffices to prove that for all $(x,y)\in S$ and for all $z\in \mathsf Z (H)$ with $\pi(xz)=\pi(yz)$, there exists an $M$-chain concatenating $xz$ and $yz$. Assume to the contrary that this does not hold and let $(x,y)\in S$ be a counter example for which $|x|+|y|$ is minimal. Let $z\in \mathsf Z (H)$ with $\pi(xz)=\pi(yz)$.
If 	$(x,y)\in S^*$, then $\mathsf d(xz,yz)=\mathsf d(x,y)\le M$, a contradiction.
Thus $(x,y)\not\in S^*$ and hence there exists $(x',y')\in S\setminus\{(1,1),(x,y)\} $ with $(x',y')\mid_S (x,y)$
such that $\pi(x'z)=\pi(y'z)$. Then $|x'|+|y'|<|x|+|y|$ and $|xx'^{-1}|+|yy'^{-1}|<|x|+|y|$ imply that there exist an $M$-chain concatenating $xz=x'(xx'^{-1})z$ and $y'(xx'^{-1})z$ and an $M$-chain concatenating $y'(xx'^{-1})z$ and $y'(yy'^{-1})z=yz$, a contradiction.

2. We refer to  \cite[Theorem 3.6]{F-G-K-T17}, and for 3. see \cite[Theorem 4.4.11]{Ge-HK06a}.
\end{proof}

These finiteness results for finitely generated monoids give rise to a core question in the area.

\smallskip
\noindent
\hypertarget{A}{{\bf Problem A.}} {\it Take a class $\mathcal C$ of distinguished objects (e.g.,  the class of Noetherian domains or the class
\phantom{{\bf Problem B.}} of Krull monoids). Provide an algebraic characterization of the  objects in $\mathcal C$ satisfying all \phantom{{\bf Problem B.}} resp. some of arithmetical finiteness properties of finitely generated monoids.}

\smallskip
\noindent
There are such algebraic characterizations of arithmetical finiteness properties in the literature (e.g.,  the finiteness of the elasticity is characterized within the class of finitely generated domains in \cite{Ka05a}; see also \cite{Ka16b}). But Problem A addresses a field of problems, many of which  are wide open.
In this survey, we show that  transfer Krull monoids of finite type satisfy the same arithmetical finiteness properties as given in Theorem \ref{3.1} (Theorem \ref{5.5}) and we characterize the finiteness of unions of sets of lengths in the setting of weakly Krull monoids (Theorems \ref{7.2} and \ref{7.4}). It is no surprise  that none of the statements of Theorem \ref{3.1} needs to hold true for general BF-monoids and Remark \ref{5.7} gathers some most striking examples.

\section{Krull monoids} \label{4}

\begin{definition} \label{4.1}
A monoid $H$ is a {\it Krull monoid} if it is cancellative  and satisfies one of the following equivalent conditions{\rm \,.}
\begin{enumerate}
\item[(a)] $H$ is a completely integrally closed Mori monoid.

\item[(b)] $H$ has a divisor theory $\partial \colon H \to \mathcal F (P)$; this means that $\partial$ is a divisor homomorphism such that for every $\alpha \in \mathcal F (P)$ there are $a_1, \ldots, a_m \in H$ with $\alpha = \gcd \bigl( \partial (a_1), \ldots, \partial (a_m) \bigr)$.

\item[(c)] $H$ has a divisor homomorphism into a free abelian monoid.
\end{enumerate}
\end{definition}

Let $H$ be a Krull monoid. Then the monoid $\mathcal I_v^* (H)$ is free abelian, and there is a free abelian monoid $F = \mathcal F (P)$ such that the inclusion $H_{\red} \hookrightarrow F$ is a divisor theory. Since divisor theories of a monoid are unique up to isomorphisms, the group
\[
\mathcal C (H) = \mathsf q (F)/\mathsf q (H_{\red})
\]
depends only on $H$ and it is called the {\it (divisor) class group} of $H$. Every $g \in \mathcal C (H)$ is a subset of $\mathsf q (F)$, $P \cap g$ is the set of prime divisors lying in $g$, and $G_0 = \{[p] = q \mathsf q (H_{\red}) \colon p \in P \} \subset \mathcal C (H)$ is the set of classes containing prime divisors.

\begin{example}[\bf Examples of Krull monoids] \label{4.2}~

1. Domains. A Noetherian domain is Krull if and only if it is integrally closed and the integral closure of any Noetherian domain is Krull (Theorem of Mori-Nagata). The property of being a Krull domain is a purely multiplicative one. Indeed, a   domain $R$ is a Krull domain if and only if its multiplicative monoid of nonzero elements is a Krull monoid (this characterization generalizes to rings with zero-divisors, see \cite[Theorem 3.5]{Ge-Ra-Re15c}).  If $R$ is a Krull domain, then $\mathcal C (R^{\bullet}) \cong \mathcal C_v (R)$, where $\mathcal C_v (R)$ is the usual $v$-class group of a Krull domain. If $R$ is a Dedekind domain, then $\mathcal C (R^{\bullet}) \cong \Pic (R)$.

2. Submonoids of domains. Since the composition of divisor homomorphisms is a divisor homomorphism, every saturated submonoid $H \subset R^{\bullet}$ of a Krull domain $R$ is a Krull monoid. But also non-Krull domains may have submonoids that are Krull. We mention two classes of examples.

Let $\mathcal O$ be an order in an algebraic number field $K$ with conductor $\mathfrak f = (\mathcal O \DP \widehat{\mathcal O})$. Then $\widehat{\mathcal O} = \mathcal O_K$ is the ring of integers of $K$ and $H = \{ a \in \mathcal O \colon a\mathcal O+ \mathfrak f = \mathcal O_K \} \subset \mathcal O^{\bullet}$ is a submonoid. Moreover, $H$ is Krull and it is an example of a regular congruence monoid (\cite[Section 2.11]{Ge-HK06a}).

Let $R$ be an integral domain with quotient field $K$. Then $\Int (R) = \{f \in K[X] \colon f (R) \subset R\} \subset K[X]$ is the ring of integer-valued polynomials. If $R$ is factorial, then the divisor-closed submonoid $\LK f \RK \subset \Int (R)$ is a Krull monoid for every nonzero polynomial $f \in \Int (R)$ (\cite{Re14a}).

3. Normal affine monoids. Let $H$ be a reduced monoid. Then $H$ is normal and affine if and only if $H$ is a finitely generated Krull monoid, which holds if and only if it is isomorphic to the monoid of non-negative solutions of a system of linear diophantine equations (\cite[Theorem 2.7.14]{Ge-HK06a}). Normal affine monoids and the associated monoid algebras play a crucial role  in combinatorial commutative algebra (\cite{Br-Gu09a}).

4. Monoids of modules. Let $R$ be a not necessarily commutative ring, $\mathcal C$ a class of $R$-modules, and $\mathcal V ( \mathcal C)$ the semigroup of modules as introduced in Example \ref{2.1}.2. By a path breaking result of Facchini (\cite[Theorem 3.4]{Fa02}), $\mathcal V ( \mathcal C)$ is a Krull monoid if the endomorphism rings $\End_R (M)$ are semilocal for all modules $M$ of $\mathcal C$ (for modules having semilocal endomorphism rings see \cite{Fa06b}). This result paved the way for studying direct-sum decomposition of modules with methods from the factorization theory of Krull monoids.

5. Monoids of zero-sum sequences. Let $G$ be an abelian group, $G_0 \subset G$ a subset, and $\mathcal F (G_0)$ the free abelian monoid with basis $G_0$. According to the tradition of additive combinatorics, elements of $\mathcal F (G_0)$ are called {\it sequences} over $G_0$. If $S = g_1 \cdot \ldots \cdot g_{\ell} \in \mathcal F (G_0)$, then $\sigma (S) = g_1 + \ldots + g_{\ell} \in G$ is the sum of $S$ and $S$ is called a {\it zero-sum sequence} if $\sigma (S)=0$.  The set $\mathcal B (G_0) = \{ S \in \mathcal F (G_0) \colon \sigma (S)=0 \} \subset \mathcal F (G_0)$ is a submonoid (called the {\it monoid of zero-sum sequences} over $G_0$) and since the inclusion $\mathcal B (G_0) \hookrightarrow \mathcal F (G_0)$ is a divisor homomorphism, $\mathcal B (G_0)$ is a Krull monoid. Suppose that  $G_0$ is finite. Then $\mathcal B (G_0)$ is finitely generated and the converse holds if $G = [G_0]$. Moreover, since $\mathcal B (G_0)$ is reduced and its quotient group is torsion-free, it is a normal affine monoid.

6. Analytic monoids. These are Krull monoids with finite class group and a suitable norm function that allows to establish a theory of $L$-functions. Analytic monoids serve as a general frame for a quantitative theory of factorizations. Let $\partial \colon H \to \mathcal F (P)$ be a divisor theory of $H$ and let $\mathsf N  \colon \mathcal F (P) \to \N$ be a norm. The goal of quantitative factorization theory is to study, for a given arithmetical property $\mathsf P$, the asymptotic behavior, for $x \to \infty$, of the associated counting function
\[
\mathsf P (x) =  \#\{ a \in H \colon \mathsf N (a) \le x, a \ \text{satisfies Property} \ \mathsf P \} \,.
\]
A systematic study of counting functions (in the setting of algebraic number fields) was initiated by Narkiewicz in the 1960s (we refer to the presentations in the monographs \cite[Chapter 9]{Na04}, \cite[Chapter 9]{Ge-HK06a}), and for recent work  to \cite{Ka17a}).
Among others, the property that  ''$\max \mathsf L (a) \le k$" was studied for every $k \in \N$. Note that $\max \mathsf L (a) = 1$ if and only if $a$ is irreducible whence $\mathsf P_1 (x)$ counts the number of irreducibles with norm $\mathsf N (a) \le x$.
The property that ''$\mathsf L (a)$ is an interval" deserves special attention. It turned out that almost all sets of lengths are intervals. More precisely, for the ring of integers $\mathcal O_K$ in an algebraic number field $K$  we have (\cite[Theorem 9.4.11]{Ge-HK06a})
\begin{equation} \label{density}
\lim_{x \to \infty} \frac{ \#\{ a\mathcal O_K  \colon \mathsf N (a) \le x,  \ \text{$\mathsf L (a)$ is an interval}  \} }{ \#\{ a\mathcal O_K  \colon \mathsf N (a) \le x \} }  = 1 \,.
\end{equation}
This result is in contrast to Theorem \ref{5.6}.3 demonstrating the  variety of sets of lengths in Krull monoids with class group $G$ and should also be compared with \hyperlink{C}{{Problem C.}} in Section \ref{6}.

\end{example}

Let $H$ be a Krull monoid, $H_{\red} \hookrightarrow F = \mathcal F (P)$ a divisor theory, $G$ an abelian group, and $(m_g)_{g \in G}$ a family of cardinal numbers. We say that $H$ has {\it characteristic} $(G, (m_g)_{g \in G})$ if there is a group isomorphism $\phi \colon G \to \mathcal C (H)$ such that $\card (P \cap \phi (g) ) = m_g$ for all $g \in G$.

\begin{theorem}[\bf Structure and Realization Results for Krull monoids] \label{4.3}~

\begin{enumerate}
\item If $G$ is an abelian group, $(m_g)_{g \in G}$ a family of cardinal numbers, $G_0 = \{ g \in G \colon m_g \ne 0 \}$, and $G_1 = \{g \in G \colon m_g=1 \}$, then the following statements are equivalent{\rm \,.}
    \begin{enumerate}
    \item There exists a Krull monoid with characteristic $(G, (m_g)_{g \in G})$.

    \item $G = [G_0]$ and $G = [G_0 \setminus \{g\}]$ for every $g \in G_1$.
    \end{enumerate}
    Moreover, two reduced Krull monoids are isomorphic if and only if they have the same characteristic.

\item For every Krull monoid $H$ there is a reduced Krull monoid $H_0$ with $H_0 \cong H_{\red}$ such that $H = H^{\times} \time H_0$.

\item For every reduced Krull monoid $H$ there is an abelian group $G$ and a subset $G_0 \subset G$ such that $H \cong \mathcal B (G_0)$.

\item For every reduced Krull monoid $H$ there is a ring $R$ and a class of $R$-modules $\mathcal C$ such that $H \cong \mathcal V (\mathcal C)$.
\end{enumerate}
\end{theorem}

\begin{proof}
For 1. - 3. see \cite[Sections 2.4 and  2.5]{Ge-HK06a} and for 4. see \cite[Theorem 2.1]{Fa-Wi04}.
\end{proof}

Next we introduce transfer homomorphisms, a key tool in factorization theory (for transfer homomorphisms in more general settings see \cite{Ba-Sm15, Fa-Tr18a}).

\begin{definition} \label{4.4}
A monoid homomorphism $\theta \colon H \to B$ between atomic monoids is said to be a {\it transfer homomorphism} if the following two properties are satisfied.

\begin{enumerate}
\item[{\bf (T\,1)\,}] $B = \theta(H) B^\times$ \ and \ $\theta
^{-1} (B^\times) = H^\times$.

\item[{\bf (T\,2)\,}] If $u \in H$, \ $b,\,c \in B$ \ and \ $\theta
(u) = bc$, then there exist \ $v,\,w \in H$ \ such that \ $u = vw$, \
$\theta (v) \simeq b$ \ and \ $\theta (w) \simeq c$.
\end{enumerate}
\end{definition}

Thus transfer homomorphisms are surjective up to units and they allow to lift factorizations. The next proposition shows that they allow one to pull back arithmetical information to the source monoid.

\begin{proposition}  \label{4.5}
Let $\theta \colon H \to B$ be a transfer homomorphism between atomic monoids.
\begin{enumerate}
\item For every $a \in H$ we have $\mathsf L_H (a) = \mathsf L_{B} \big( \theta (a) \big)$. In particular, we have $\mathcal L (H) = \mathcal L (B)$, $\Delta (H) = \Delta (B)$, $\mathcal R (H) = \mathcal R (B)$, and $\mathcal U_k (H) = \mathcal U_k (B)$ for all $k \in \N$.

\item $\mathsf c (B) \le \mathsf c (H) \le \max \{\mathsf c (B), \mathsf c (H, \theta)\}$, where $\mathsf c (H, \theta)$ is the catenary degree in the fibres.
\end{enumerate}
\end{proposition}

\begin{proof}
1. It follows easily from the definition that $\mathsf L_H (a) = \mathsf L_{B} \big( \theta (a) \big)$ for every $a \in H$ (for a proof in the cancellative setting see \cite[Proposition 3.2.3]{Ge-HK06a}). The remaining statements are an immediate consequence.

2. The proof runs along the same lines as in the cancellative setting (\cite[Theorem 3.2.5]{Ge-HK06a}).
\end{proof}

\begin{proposition} \label{4.6}
Let $H$ be a reduced Krull monoid, $F = \mathcal F (P)$ be a free abelian monoid such that the inclusion $H \hookrightarrow F$ is a cofinal divisor homomorphism. Let $G = \mathsf q (F)/\mathsf q (H)$ be the class group, $G_0 = \{[p] = p \mathsf q (H) \colon p \in P\} \subset G$ denote the set of classes containing prime divisors, and let $\widetilde{\boldsymbol \beta} \colon F \to \mathcal F (G_0)$ be the unique homomorphism such that $\widetilde{\boldsymbol \beta} (p) = [p]$ for all $p \in P$.
\begin{enumerate}
\item For every $a \in F$, we have $a \in H$ if and only if $\widetilde{\boldsymbol \beta} (a) \in \mathcal B (G_0)$.

\item The restriction $\boldsymbol \beta = \widetilde{\boldsymbol \beta}|H \colon H \to \mathcal B (G_0)$ is a transfer homomorphism with $\mathsf c (H, \boldsymbol \beta) \le 2$.
\end{enumerate}
\end{proposition}

\begin{proof}
1. Let $a = p_1 \cdot \ldots \cdot p_{\ell} \in F$, where $\ell \in \N_0$ and $p_1, \ldots, p_{\ell} \in P$. Since the inclusion $H \hookrightarrow F$ is a divisor homomorphism, we have $H = \mathsf q (H) \cap F$, whence $a \in H$ if and only if $0 = [a] = [p_1] + \ldots + [p_{\ell}] = \sigma ( \widetilde{\boldsymbol \beta} (a) )$.

2. By 1., we have $\boldsymbol \beta (H) = \mathcal B (G_0)$ whence $\boldsymbol \beta \colon H \to \mathcal B (G_0)$ is surjective and $\boldsymbol \beta^{-1} (1) = \{1\}$. To verify {\bf (T2)}, let $a = p_1 \cdot \ldots \cdot p_{\ell} \in H$ be given with $\ell \in \N_0$ and $p_1, \ldots, p_{\ell} \in P$. Suppose that $\boldsymbol \beta (a) = BC$ with $B, C \in \mathcal B (G_0)$, say $B = [p_1] \cdot \ldots \cdot [p_k]$ and $C = [p_{k+1}] \cdot \ldots \cdot [p_{\ell}]$ with $k \in [0, \ell]$. Then 1. implies that $b = p_1 \cdot \ldots \cdot p_k \in H$, $c = p_{k+1} \cdot \ldots \cdot p_{\ell} \in H$, and clearly $a=bc$. Thus $\boldsymbol \beta$ is a transfer homomorphism. The inequality $\mathsf c (H, \boldsymbol \beta) \le 2$ follows from \cite[Proposition 3.4.8]{Ge-HK06a}.
\end{proof}

\section{Transfer Krull monoids} \label{5}

Within the class of Mori monoids, Krull monoids are the ones whose arithmetic is best understood.
Transfer Krull monoids need not be Krull but they have the same arithmetic as Krull monoids.
They include all commutative Krull monoids, but also classes of not integrally closed Noetherian domains and of non-commutative Dedekind domains (see Example \ref{5.4}). We start with the definition, discuss some basic properties, and as a main structural result we show that for every cancellative transfer Krull monoid there is an overmonoid that is Krull such that the inclusion is a transfer homomorphism (Proposition \ref{5.3}.2).

\begin{definition} \label{5.1}
A monoid $H$ is said to be a {\it transfer Krull monoid} if one of the following two equivalent properties is satisfied{\rm \,.}
\begin{enumerate}
\item[(a)] There exist a Krull monoid $B$ and a transfer homomorphism $\theta \colon H \to B$.

\item[(b)] There exist an abelian group $G$, a subset $G_0 \subset G$, and a transfer homomorphism $\theta \colon H \to \mathcal B (G_0)$.
\end{enumerate}
\end{definition}

Since $\mathcal B (G_0)$ is a Krull monoid by Example \ref{4.2}.5, Property (b) implies Property (a). Conversely, since every Krull monoid has a transfer homomorphism to a monoid of zero-sum sequences by Proposition \ref{4.6} and since the composition of transfer homomorphisms is a transfer homomorphism, Property (a) implies Property (b). Thus Property (a) and Property (b) are equivalent. If $H$ satisfies Property (b) with a finite set $G_0$, then $H$ is said to be {\it transfer Krull of finite type}. Since Krull monoids are BF-monoids, transfer Krull monoids are BF-monoids by Proposition \ref{4.5}.1, but they need neither be Mori nor be completely integrally closed.

\begin{lemma} \label{5.2}~

\begin{enumerate}
\item Coproducts of transfer Krull monoids are transfer Krull.

\item Divisor-closed submonoids of transfer Krull monoids are transfer Krull.

\item Let $H$ be a cancellative monoid, $\theta \colon H \to B$ be a transfer homomorphism to a reduced Krull monoid $B$, $T \subset B$ be a submonoid, and $S = \theta^{-1} (T)$. Then  $\Theta = \mathsf q (\theta)| S^{-1} H \colon S^{-1}H \to T^{-1}B$ is a transfer homomorphism and $S^{-1}H$ is  transfer Krull.
\end{enumerate}
\end{lemma}

\begin{proof}
1. Let $(H_i)_{i \in I}$ be a family of transfer Krull monoids and
\[
H = \coprod_{i \in I} H_i = \Big\{ (a_i)_{i \in I} \in \prod_{i \in I} H_i \colon a_i = 1 \ \text{for almost all} \ i \in I \Big\}
\]
their coproduct. If $(\theta_i \colon H_i \to B_i)_{i \in I}$ is a family of transfer homomorphisms into the Krull monoids $B_i$, then the homomorphism $\theta = (\theta_i)_{i \in I} \colon H \to \coprod_{i \in I} B_i$ is a transfer homomorphism. Since the coproduct of Krull monoids is a Krull monoid, $H$ is a transfer Krull monoid.

2. Let $\theta \colon H \to B$ be a transfer homomorphism to a reduced Krull monoid $B$ and let $S \subset H$ be a divisor-closed submonoid. Then the restriction $\theta | S \colon S \to \theta (S)$ is a transfer homomorphism, $\theta (S) \subset B$ is a divisor-closed submonoid, and since divisor-closed submonoids of Krull monoids are Krull, the divisor-closed submonoid $S \subset H$ is a transfer Krull monoid.

3. Since localizations of Krull monoids are Krull, $T^{-1}B$ is a Krull monoid and hence it suffices to verify that $\Theta \colon S^{-1}H \to T^{-1}B$ is a transfer homomorphism. Since $\theta$ is surjective, we infer that $\Theta$ is surjective. An elementary calculation shows that $\Theta^{-1} \bigl( (T^{-1}B)^{\times} \bigr) = (S^{-1}H)^{\times}$. Thus {\bf (T1)} holds. In order to verify {\bf (T2)}, let $u = \frac{h}{s} \in S^{-1}H$, $b = \frac{b_1}{t_1}, c = \frac{b_2}{t_2} \in T^{-1}B$ be such that $\Theta (u) = bc$, where $h \in H$, $s \in S$, $b_1, b_2 \in B$, and $t_1,t_2 \in T$. Let $s_1,s_2 \in S$ be such that $\theta (s_1)=t_1$ and $\theta (s_2)=t_2$. Then
\[
\frac{\theta (h)}{\theta (s)} = \frac{b_1}{\theta(s_1)}\frac{b_2}{\theta (s_2)} \quad \text{whence} \quad \theta (hs_1s_2) = b_1 \bigl( b_2 \theta (s) \bigr) \,.
\]
Since $\theta \colon H \to B$ is a transfer homomorphism, there are $x,y \in H$ such that $\theta (x)=b_1$, $\theta (y) = b_2 \theta (s)$, and $hs_1s_2=xy$. Thus we obtain that
\[
u = \frac{h}{s} = \frac{x}{s_1}\frac{y}{s_2s}, \ \Theta \Bigl( \frac{x}{s_1} \Bigr) = \frac{b_1}{t_1}, \quad \text{and} \quad \Theta \Bigl( \frac{y}{s_2s} \Bigr) = \frac{b_2}{t_2} \,,
\]
and hence {\bf (T2)} holds.
\end{proof}

\begin{proposition} \label{5.3}
Let $H$ be a monoid and  $H_{\canc}$ be the associated cancellative monoid.
\begin{enumerate}
\item $H$ is a transfer Krull monoid if and only if there is a Krull monoid $D$ with $H_{\canc} \subset D \subset \mathsf q (H_{\canc})$ such that the canonical map $\Theta \colon H \twoheadrightarrow H_{\canc} \hookrightarrow D $ is a transfer homomorphism. If this holds, then $\mathsf q (H_{\canc}) = \mathsf q (D)$, $D = H_{\canc}D^{\times}$, and $H_{\canc}^{\times} = D^{\times} \cap H_{\canc}$.

\item Suppose that $H$ is cancellative. Then $H$ is a transfer Krull monoid if and only if there is a Krull monoid $D$ with $H \subset D \subset \mathsf q (H)$ such that the inclusion $H \hookrightarrow D$ is a transfer homomorphism. If this holds, then $\mathsf q (H) = \mathsf q (D)$, $D = HD^{\times}$, and $H^{\times} = D^{\times} \cap H$.

\item If $R \subset S$ are integral domains with $\mathsf q (R)=\mathsf q (S)$, $S = RS^{\times}$, $R^{\times} = S^{\times} \cap R$, and $(R \colon S) \in \max (R)$, then the inclusion $R^{\bullet} \hookrightarrow S^{\bullet}$ is a transfer homomorphism. If, in addition,  $S$ is a Krull domain, then $R$ is a transfer Krull domain.
\end{enumerate}
\end{proposition}

\begin{proof}
1. Clearly, if $D$ is a Krull monoid and $\Theta \colon H \twoheadrightarrow  H_{\canc}\hookrightarrow D$ is a transfer homomorphism, then $H$ is a transfer Krull monoid. Conversely, suppose that $H$ is a transfer Krull monoid and let $\theta\colon H\rightarrow B$ be a transfer homomorphism, where $B$ is a reduced Krull monoid. For an element $a \in H$, we denote by $[a] \in H_{\canc}$ the congruence class of $a$.
If $a_1, a_2, c \in H$ such that $a_1c=a_2c$, then $\theta (a_1) \theta (c) = \theta (a_2) \theta (c)$ whence $\theta (a_1)=\theta (a_2)$.
Thus $\theta$ induces a homomorphism  $\theta^*\colon  H_{\canc}\rightarrow B$,  defined by $\theta^*([a])=\theta(a)$ for all $a\in H$.
Since $\theta$ is a transfer homomorphism, it is easy to see $\theta^*$ is a transfer homomorphism.
	
	 If $D=\{[a]^{-1}[b]  \colon a,b\in H, \theta(a)\mid_{B} \theta(b)\} \subset \mathsf q(H_{\canc})$, then the  homomorphism $\mathsf q (\theta^*)|D \colon D \rightarrow B$  is a divisor homomorphism, whence $D$ is a Krull monoid. By construction, we have $H_{\canc} \subset D\subset \mathsf q(H_{\canc})$ and thus $\mathsf q (H_{\canc})=\mathsf q (D)$.
	
To verify that $\Theta \colon H \twoheadrightarrow H_{\canc} \hookrightarrow D$ is a transfer homomorphism, we first note that $D^{\times} = \{[a]^{-1}[b] \colon a,b \in H \ \text{with} \  \theta (a) = \theta (b)\}$. Now let $[a]^{-1}[b] \in D$ where $a,b\in H$  with $\theta(a)\mid_{B} \theta(b)$. Then there exists $c \in B$ such that $\theta(b)=\theta(a)c$. Since $\theta$ is a transfer homomorphism, there exist $b_1,b_2\in H$ such that $b=b_1b_2$ and $\theta(b_1)=\theta(a)$, $\theta(b_2)=c$. It follows that $[a]^{-1}[b_1] \in D^{\times}$ and $[a]^{-1}[b] =([b_1]^{-1}[b_1b_2])([a]^{-1}[b_1])=[b_2]([a]^{-1}[b_1]) \in H_{\canc}D^{\times}$, whence $D=H_{\canc}D^{\times} = \Theta (H)D^{\times}$. Similarly, we get $H_{\canc}^{\times} = D^{\times} \cap H_{\canc}$ whence {\bf (T1)} holds.
	
To verify {\bf (T2)}, let $a\in H$ and $d_1,d_2\in D$ be given such that $[a]=d_1d_2$. Then $\theta(a)=\theta^*(d_1)\theta^*(d_2)$ and hence there exist $a_1,a_2\in H$ such that $a=a_1a_2$ and $\theta^*([a_1])=\theta^*(d_1)$, $\theta^*([a_2])=\theta^*(d_2)$. It follows that $[a_1]=d_1 (d_1^{-1}[a_1]) \in d_1D^{\times}$ and $[a_2]=d_2 (d_2^{-1}[a_2]) \in d_2D^{\times}$. Therefore $\Theta$ is a transfer homomorphism.		

2.  is a special case of 1. and for 3. we refer to  \cite[Proposition 3.7.5]{Ge-HK06a}.
\end{proof}

\begin{example}[\bf Examples of transfer Krull monoids] \label{5.4}~

1. Since the identity map is a transfer homomorphism, every Krull monoid is a transfer Krull monoid. This generalizes to not necessarily commutative, but normalizing Krull monoids as studied in the theory of Noetherian semigroup algebras (\cite{Je-Ok07a, Ge13a, Ok16a}).

2. Every half-factorial monoid is transfer Krull. Indeed, let $H$ be half-factorial and let $G = \{0\}$ be the trivial group. Then $\theta \colon H \to \mathcal B (G)$, defined by $\theta (u) = 0$ for every $u \in \mathcal A (H)$ and $\theta (\epsilon) = 1$ for every $\epsilon \in H^{\times}$, is a transfer homomorphism.

3. Main examples of transfer Krull monoids stem from non-commutative ring theory whence they are beyond the scope of this article. Nevertheless, we mention one example and refer the interested reader to \cite{Ba-Ba-Go14, Ba-Sm15, Sm16a, Sm19a} for more. Let $\mathcal O_K$ be the ring of integers in an algebraic number field $K$, $A$ a central simple $K$-algebra, and $R$ a classical maximal $\mathcal O_K$-order of $A$. Then $R^{\bullet}$ is transfer Krull if and only if every stably free left $R$-ideal is free, and if this holds then there is a transfer homomorphism $\theta \colon R^{\bullet} \to \mathcal B (G)$ for some finite abelian group $G$ (\cite{Sm13a}).

4. The assumptions made in Proposition \ref{5.3}.3 hold true for $(K+\mathfrak m)$-domains (\cite[Proposition 3.7.4]{Ge-HK06a}). Further applications in the setting of seminormal weakly Krull monoids and domains are given in  (\cite[Proposition 4.6 and Theorem 5.8]{Ge-Ka-Re15a}.

5. Module theory offers a wealth of non-cancellative finitely generated transfer Krull monoids. We discuss a simple example. Let $B$ be an additive Krull monoid with $\mathcal A (B) = \{u_1, u_2, v\}$ such that $u_1+u_2=v+v+v$ is the only relation among the atoms. Let $H$ be the free abelian monoid with basis $\{M_1, M_2, M_2', Q\}$ modulo the relation generated by $M_1+M_2=M_1+M_2'=Q+Q+Q$. By the Theorem of Bergman-Dicks (see Example \ref{2.1}.1), $H$ is isomorphic to a monoid of modules $\mathcal V ( \mathcal C)$, where $\{M_1, M_2, M_2', Q\}$ is a set of representatives of indecomposable modules in $\mathcal C$. Clearly, $\theta \colon H \to B$, defined by $\theta (M_1)=u_1$, $\theta (M_2)=\theta (M_2')=u_2$, and $\theta (Q)=v$, is a transfer homomorphism.
\end{example}

\begin{theorem}[\bf Arithmetic of transfer Krull monoids] \label{5.5}
Let $H$ be a transfer Krull monoid of finite type.
Then $\mathcal R (H) = \{ q \in \Q \colon 1 \le q \le \rho (H)\}$ and all arithmetical finiteness results of Theorem \ref{3.1} hold.
\end{theorem}

\begin{proof}
Let $\theta \colon H \to \mathcal B (G_0)$ be a transfer homomorphism where $G_0$ is a finite subset of an abelian group. Since $G_0$ is finite, $\mathcal B (G_0)$ is finitely generated, whence the finiteness results of Theorem \ref{3.1} hold for $\mathcal B (G_0)$ and they can be pulled back to $H$ by Proposition \ref{4.5}. The claim on $\mathcal R (H)$ follows from \cite[Theorem 3.1]{Ge-Zh19a}.
\end{proof}

Our next theorem shows that, for the class of finitely generated Krull monoids, the finiteness result for the set of distances and for the set of catenary degrees, as well as the structural result for sets of lengths (given in Theorems \ref{3.1} and \ref{5.5}), are best possible.

\begin{theorem}[\bf Realization Results] \label{5.6}~

\begin{enumerate}
\item For every finite nonempty subset $C \subset \N_{\ge 2}$ there is a finitely generated Krull monoid $H$ with finite class group such that $\Ca (H) = C$.

\item For every finite nonempty set $\Delta \subset \N$ with $\min \Delta  = \gcd \Delta$ there is a finitely generated Krull monoid $H$ such that $\Delta (H) = \Delta$.

\item For every $M \in \N_0$ and every finite nonempty set $\Delta$ there is a finitely  generated Krull monoid $H$ with finite class group such that the following holds: for every {\rm AAMP} L with difference $d \in \Delta$ and bound $M$ there is some $y^{}_L \in \N$ such that $y+L \in \mathcal L (H)$ for all $y \ge y^{}_L$.
\end{enumerate}
\end{theorem}

\begin{proof}
For 1. we refer to \cite[Proposition 3.2]{Fa-Ge19a}, for 2.  to \cite{Ge-Sc17a}, and for 3. see \cite{Sc09a}.
\end{proof}

\begin{remark} \label{5.7}~

Each of the following monoids respectively domains has the property that every finite nonempty subset of $\N_{\ge 2}$ occurs as a set of lengths.

\begin{itemize}
\item (Frisch) The ring $\Int (\Z)$ of integer-valued polynomials over $\Z$ (\cite{Fr13a, Fr-Na-Ri19a}).

\item (Kainrath) Krull monoids with infinite class group and prime divisors in all classes (\cite{Ka99a} and \cite[Theorem 7.4.1]{Ge-HK06a}).
\end{itemize}
The assumption, that every class contains a prime divisor, is crucial in Kainrath's Theorem. Indeed, on the other side of the spectrum, there is the conjecture that every abelian group is the class group of a half-factorial Krull monoid (even of a half-factorial Dedekind domain; \cite{Gi06a}).
According to a conjecture of Tringali, the power monoid $\mathcal P_{\fin, 0} (\N_0)$ (and hence the monoid $\mathcal P_{\fin} (\N_0)$) has the property that every finite nonempty subset of $\N_{\ge 2}$ occurs as a set of lengths. This conjecture is supported by a variety of results such as $\Ca \big( \mathcal P_{\fin,0} (\N_0) \big) = \Delta \big( \mathcal P_{\fin, 0} (\N_0) \big) = \N$ (\cite[Theorem 4.11]{Fa-Tr18a}).
\end{remark}

Thus both extremal families,
\[
\bigl\{ \{k\} \colon k \in \N_0 \bigr\} \subset \mathcal P_{\fin} (\N_0) \quad \text{and} \quad \bigl\{ \{0\}, \{1\} \bigr\} \cup \mathcal P_{\fin} ( \N_{\ge 2}) \subset \mathcal P_{\fin} (\N_0) \,,
\]
are systems of sets of lengths of BF-monoids.
Clearly, every subset $\mathcal L \subset \mathcal P_{\fin} (\N_0)$, that is the system of sets of lengths of a BF-monoid $H$  (i.e., $\mathcal L = \mathcal L (H)$) with $H \ne H^{\times}$,  has the following properties.
\begin{itemize}
\item[(a)] $\{0\}, \{1\} \in \mathcal L $ and  all other sets of $\mathcal L $ lie in $\N_{\ge 2}$.

\item[(b)] For every $k \in \N_0$ there is  $L \in \mathcal L $ with $k \in L$.

\item[(c)] If $L_1, L_2 \in \mathcal L $, then there is $L \in \mathcal L $ with $L_1+L_2 \subset L$.
\end{itemize}

\noindent
This gives rise to the following realization problem.

\smallskip
\noindent
\hypertarget{B}{{\bf Problem B.}}  {\it Which subsets $\mathcal L \subset \mathcal P_{\fin} (\N_0)$  satisfying Properties (a) - (c)  can be realized as systems of sets \phantom{{\bf Problem A.}} of lengths of a \BF-monoid?}

\smallskip
\noindent
Note that   every system $\mathcal L$ with (a) - (c) and with $\Delta (\mathcal L) \ne \emptyset$  satisfies the property $\min \Delta ( \mathcal L) = \gcd \Delta ( \mathcal L)$ (\cite[Proposition 2.9]{F-G-K-T17}), which holds for all systems stemming from BF-monoids.

We end this section with a list of monoids and domains that are not transfer Krull and we will discuss such monoids in Section \ref{7}.

\begin{example}[\bf Monoids and domains that are not transfer Krull] \label{5.8}~

1. According to Remark \ref{5.7} a Krull monoid with infinite class group having prime divisors in all classes and $\Int (\Z)$ have the same system of sets of lengths. Nevertheless, $\Int (\Z)$ is not transfer Krull (\cite{Fr-Na-Ri19a}). Similarly, the monoid of polynomials having nonnegative integer coefficients is not transfer Krull (\cite[Remark 5.4]{Ca-Fa19a}).

2. $\mathcal P_{\fin} (\N_0)$ and $\mathcal P_{\fin, 0} (\N_0)$ are reduced BF-monoids that have no transfer homomorphism to any cancellative monoid (\cite[Proposition 4.12]{Fa-Tr18a}).

3. Let $G$ be a finite group and let $\mathcal B (G)$ be the monoid of product-one sequences over $G$. If $G$ is abelian, then $\mathcal B (G)$ is Krull by Example \ref{4.2}.5 and hence transfer Krull. Jun Seok Oh showed that $\mathcal B (G)$ is transfer Krull if and only if it is  Krull if and only if $G$ is abelian (\cite[Proposition 3.4]{Oh19b}).

4. An additive submonoid of the nonnegative rational numbers (distinct from $\{0\}$) is transfer Krull if and only if it is isomorphic to $(\N_0, +)$  (\cite[Theorem 6.6]{Go18a}).

5. In Section \ref{7} we show that strongly primary monoids and monoids of ideals of weakly Krull monoids are  transfer Krull if and only if they are half-factorial (Lemma \ref{7.1} and Proposition \ref{7.3}).
\end{example}

\section{Transfer Krull monoids over finite abelian groups} \label{6}

In this section we discuss transfer Krull monoids $H$ having a transfer homomorphism $\theta \colon H \to \mathcal B (G)$, where $G$ is a finite abelian group. By Proposition \ref{4.6}, this setting includes Krull monoids with finite class groups having prime divisors in all classes. Rings of integers of algebraic number fields (\cite[Corollary 4.4.3]{HK20a}), monoid algebras that are Krull (\cite{Ch11a}), and many other Krull domains have finite class group and prime divisors in all classes. This is the reason why this setting has received the closest attention in factorization theory.

Let $H$ be a transfer Krull monoid over a finite abelian group $G$, say $G \cong C_{n_1} \oplus \ldots \oplus C_{n_r}$ with $1 \le n_1 \t \ldots \t n_r$. It is usual to write $* (G)$ instead of $* ( \mathcal B (G))$ for all invariants we had. In particular, we set
\[
\mathcal L (G) := \mathcal L \big( \mathcal B (G) \big), \Delta (G) := \Delta \bigl( \mathcal B (G) \bigr), \mathcal A (G) := \mathcal A \big( \mathcal B (G) \big),   \quad \text{and so on}.
\]
By Propositions \ref{4.5} and \ref{4.6}, the arithmetical invariants  of $H$ and of $\mathcal B (G)$ coincide (except for some trivial exceptions), whence $\mathcal L (H) = \mathcal L (G)$ and so on. The long term goal is to determine the precise value of  these invariants
in terms of  the group invariants $(n_1, \ldots, n_r)$, which is done with methods from additive combinatorics. We refer to \cite[Chapter 1]{Ge-Ru09} for a detailed discussion of the interplay of factorization theory in $\mathcal B (G)$ and additive combinatorics  and  to the survey \cite{Sc16a} for the state of the art. We have a quick glance at this interplay, introduce a key combinatorial invariant,  and present a main problem.

Since the group $G$ is finite, $\mathcal B (G)$ is finitely generated whence $\mathcal A (G)$ is finite, and the {\it Davenport constant} $\mathsf D (G)$, defined as
\[
\mathsf D (G) = \max \{ |U| \colon U \in \mathcal A (G) \}
\]
is a positive integer. The Davenport constant and the structure of atoms $U \in \mathcal A (G)$ with $|U|=\mathsf D (G)$ play an important role in all arithmetical investigations. So we have, for example,  $\rho (G) = \mathsf D (G)/2$ and  $\mathsf c (G) \le \mathsf D (G)$. It needs just a few lines to verify that
$1 + \sum_{i=1}^r (n_i-1) \le \mathsf D (G) \le |G|$,
whence $\mathsf D (G)=|G|$ if $G$ is cyclic. If  $G$ is a $p$-group or has rank $r \le 2$, then $\mathsf D (G) = 1 + \sum_{i=1}^r (n_i-1)$ but this equality does not hold in general. The precise value of $\mathsf D (G)$ is unknown even for rank three groups and for groups of the form $G=C_n^{r}$, where $r, n \in \N$  (\cite{Gi18a}).

What do we know about $\mathcal L (H) = \mathcal L (G)$? It is easy to verify that
\begin{equation} \label{system-1}
\mathcal L (C_1) = \mathcal L (C_2) = \big\{ \{k\} \colon k \in \N_0 \big\} \qquad \text{and  that}
\end{equation}
\begin{equation} \label{system-2}
\mathcal L (C_3) = \mathcal L (C_2 \oplus C_2) = \bigl\{ y
      + 2k + [0, k] \, \colon \, y,\, k \in \N_0 \bigr\} \,.
\end{equation}
The above four groups are precisely the groups $G$ having Davenport constant $\mathsf D (G) \le 3$. Apart from them, the systems $\mathcal L (G)$ are also written down explicitly   for all groups $G$ with  $\mathsf D (G) \in [4, 5]$ (\cite{Ge-Sc-Zh17b}). Full descriptions of systems $\mathcal L (G)$ are hard to get, whence  the focus of research is to get at least a good understanding for parameters controlling sets of lengths.
We cite one result and this is in sharp contrast to Theorem \ref{5.6}.

\begin{theorem} \label{6.1}
Let $G$ be a finite abelian group.
\begin{enumerate}
\item $($Carlitz $1960)$ $|L|=1$ for every $L \in \mathcal L (G)$ if and only if $|G| \le 2$.

\item The unions $\mathcal U_k (G)$ are finite intervals for every $k \in \N$.

\item The set of distances $\Delta (G)$ and the set of catenary degrees $\Ca (G)$ are finite intervals.
\end{enumerate}
\end{theorem}

\begin{proof}
1. If $|G| \le 2$, then $\mathcal B (G)$ is factorial whence half-factorial. Conversely, suppose that $|G|\ge 3$. If $g \in G$ with $\ord (g)=n \ge 3$, then $U=g^n$, $-U = (-g)^n$, and $V = (-g)g$ are atoms of $\mathcal B (G)$. Then  $U(-U)=V^n$ and $\mathsf L (V^n)= \{2,n\}$. If $e_1, e_2 \in G$ with $\ord (e_1)=\ord (e_2)=2$, then $e_0 = e_1+e_2 \in G$, $V = e_0e_1e_2\in\mathcal A(G)$, and $U_i = e_i^2 \in \mathcal A (G)$ for $i \in [0,2]$. Then $U_0U_1U_2=V^2$ and $\mathsf L (V^2) = [2,3]$.

2. and 3. The unions $\mathcal U_k (G)$ are intervals by \cite[page 36]{Ge-Ru09} and 3. follows from \cite[Theorem 4.1]{Ge-Zh19a}.
\end{proof}

Suppose that $\mathsf D (G) \ge 4$. The minima of the sets $\mathcal U_k (G)$ can be expressed in terms of their maxima, and for the maxima $\rho_k (G) = \max \mathcal U_k (G)$ we have the following. For every $k \in \N$,  $\rho_{2k}(G)=k \mathsf D (G)$, and $k \mathsf D (G) + 1 \le \rho_{2k+1} (G) \le k \mathsf D (G) + \mathsf D (G)/2$ (for all this and for  more on $\rho_{2k+1}(G)$ see \cite{Sc16a}). It is easy to see that $\min \Delta (G)=1$ and that $\min \Ca (G) = 2$.
The maxima of $\Delta (G)$ and of $\Ca (G)$ are known only for very special classes of groups which includes cyclic groups (\cite{Sc16a}).

To sum up our discussion so far, given a transfer Krull monoid $H$ over $G$, arithmetical invariants of $H$ depend only on $G$ (in particular, $\mathcal L (H) = \mathcal L (G)$) and the goal is to describe them in terms of the group invariants. The associated inverse problem (known as the Characterization Problem) asks whether the system $\mathcal L (G)$ is characteristic for the group. More precisely, it reads as follows.

\smallskip
\noindent
\hypertarget{C}{{\bf Problem C.}} {\it Let $G$ be a finite abelian group with Davenport constant $\mathsf D (G) \ge 4$, and let $G'$ be an abelian \phantom{{\bf Problem C.}} group with $\mathcal L (G) = \mathcal L (G')$. Are $G$ and $G'$ isomorphic?}

\smallskip
\noindent
In spite of results stating that the typical set of lengths in $\mathcal L (G)$  is an interval (e.g., see \eqref{density}), the standing conjecture is that the exceptional sets of lengths in $\mathcal L (G)$ are characteristic for the group. In other words, the conjecture is
that the above question has an affirmative answer and  we refer to \cite{Ge-Sc19a, Ge-Zh17b, Zh18a, Zh18b} and to \cite[Theorem 5.3]{Sc11b} for recent progress. Clearly, all such studies require a detailed understanding of sets of lengths in terms of the group invariants $(n_1, \ldots, n_r)$ of $G$. We address one subproblem. For any BF-monoid $H$ and two elements $a, b \in H$ the sumset $\mathsf L (a) + \mathsf L (b)$ is contained in $\mathsf L (ab)$ but, in general, we do not have equality. This is the reason why, in general, the system $\mathcal L (H)$, considered as a subset of $\mathcal P_{\fin} (\N_0)$, is not a submonoid. On the other hand, the explicit descriptions given in \eqref{system-1} and \eqref{system-2} show that $\mathcal L (C_1), \mathcal L (C_2), \mathcal L (C_3)$, and $\mathcal L (C_2 \oplus C_2)$ are submonoids of $\mathcal P_{\fin} (\N_0)$. There is a characterization of all finite abelian groups $G$ for which $\mathcal L (G)$ is a submonoid, and in the following result we show that all of them are finitely generated.

\begin{theorem} \label{6.2}
Let $G$ be a finite abelian group. Then the following statements are equivalent{\rm \,:}
\begin{enumerate}
\item[(a)]   $\mathcal L (G)$ is a submonoid of $\mathcal P_{\fin} (\N_0)$.

\item[(b)]  All sets of lengths in $\mathcal L (G)$ are arithmetical progressions.

\item[(c)] $G$ is cyclic of order $|G| \le 4$ or isomorphic to a subgroup of  $C_2^3$ or isomorphic to a subgroup of  $C_3^2$.
\end{enumerate}
If these statements hold, then  $\mathcal L (G)$ is a finitely generated submonoid of $\mathcal P_{\fin} (\N_0)$. More precisely, we have
\begin{enumerate}
\item $\mathcal L (C_1) = \mathcal L (C_2) \cong (\N_0, +)$ and $\{1\}$ is the unique  prime element of $\mathcal L (C_1) = \mathcal L (C_2)$.

\item $\mathcal L (C_3) = \mathcal L (C_2 \oplus C_2) \cong (\N_0^2, +)$ and $\{1\}, [2,3]$ are the two prime elements of $\mathcal L (C_3) = \mathcal L (C_2 \oplus C_2)$.

\item $\mathcal L (C_4)$  is a non-cancellative non-transfer Krull monoid with $\mathcal A(\mathcal L(C_4))=\big\{ \{1\}, [2,3], [3,5], \{2,4\} \big\}$.

\item $\mathcal L (C_2^3)$ is a non-cancellative  non-transfer Krull monoid containing $\mathcal L (C_4)$ and with
		$ \mathcal A(\mathcal L(C_2^3))=\{ \{1\}, [2,3], [3,5],[3,6], [4,8], \{2,4\}\}$.

\item $\mathcal L (C_3^2)$ is a cancellative non-transfer Krull monoid with
		$\mathcal A(\mathcal L(C_3^2))=\{ \{1\}, [2,3],[2,4],[2,5],[3,7] \}$.
\end{enumerate}
\end{theorem}

\begin{proof}
The equivalence of (a), (b), and (c) is proved in  \cite[Theorem 1.1]{Ge-Sc19d}, and we prove the Claims 1. - 5.

Claim 1.    follows immediately from \eqref{system-1}.
To verify Claim 2, we observe that   for $y, k \in \N_0$ we have $y
      + 2k + [0, k] = y \{1\} + k [2,3]$ whence  \eqref{system-2} shows that $\{1\}$ and $[2,3]$ generate $\mathcal L (C_3)$ and clearly both elements are atoms. To verify that they are primes, let $y,y',k,k'\in \N_0$ such that $y\{1\}+k[2,3]=y'\{1\}+k'[2,3]$. This implies that $y=y'$ and $k=k'$ whence $\mathcal L (C_3)$ is factorial and $\{1\}$ and $[2,3]$ are primes.

In order to prove 3. and 4., we use the explicit description of $\mathcal L (C_4)$ and $\mathcal L (C_2^3)$ (\cite[Theorem 7.3.2]{Ge-HK06a}). We have
\begin{itemize}
\item $\mathcal L (C_4) = \bigl\{ y + k+1 + [0,k] \, \colon\, y,
      \,k \in \N_0 \bigr\} \,\cup\,  \bigl\{ y + 2k + \{2\nu \colon \nu \in [0,k]\} \, \colon
      \, y,\, k \in \N_0 \bigr\} $.

\item $\mathcal L (C_2^3)  =  \bigl\{ y + (k+1) + [0,k] \,
      \colon\, y
      \in \N_0, \ k \in [0,2] \bigr\}$ \newline
      $\quad \text{\, } \ \qquad$ \quad $\cup \ \bigl\{ y + k + [0,k] \, \colon\, y \in \N_0, \ k \ge 3 \bigr\}
      \cup \bigl\{ y + 2k
      + \{2\nu \colon \nu \in [0,k]\} \, \colon\, y ,\, k \in \N_0 \bigr\}$.
\end{itemize}
These descriptions show that $\mathcal L (C_4) \subset \mathcal L (C_2^3)$.
Let $y, k, t\in \N_0$. Then $y+2k+\{2\nu \colon \nu \in [0,k]\}=y\{1\}+k\{2,4\}$.
	If $k=2t+1$ is odd, then $y+k+1+[0,k]=y\{1\}+[2,3]+t\{2,4\}$.
	If $k=2t+2$ is even, then  $y+k+1+[0,k]=y\{1\}+[3,5]+t\{2,4\}$.
	If $k=0$, then $y+k+1+[0,k]=\{y+1\}=(y+1)\{1\}$. Thus $\mathcal L(C_4)$ is generated by $\{1\}, [2,3],[3,5], \{2,4\}$ and all these elements are atoms. Since $[3,5]+[2,3]=\{1\}+\{2,4\}+[2,3]=[5,8]$, we obtain that $\mathcal L(C_4)$ is not cancellative.
	Since $\{y+k+[0,k]\mid y\in \N_0, k\ge 3\}$ is generated by $\{1\}, [3,6],[4,8], \{2,4\}$, we obtain $\mathcal L(C_2^3)$ is generated by $\{1\}, [2,3], [3,5], [3,6], [4,8],\{2,4\}$. It is easy to see that all these elements are atoms and as before we infer that $\mathcal L (C_2^3)$ is not cancellative.

Assume to the contrary that there is a transfer homomorphism  $\theta\colon\mathcal L(C_4)\rightarrow H$, where $H$ is a Krull monoid. Then $\theta([3,5]), \theta(\{1\}), \theta(\{2,4\})$ are atoms, $\theta ([5,8])=\theta([3,5])+\theta([2,3])=\theta(\{1\})+\theta(\{2,4\})+\theta([2,3])$ whence   $\theta([3,5])= \theta(\{1\})+\theta(\{2,4\})$, a contradiction. The same argument shows that $\mathcal L(C_2^3)$ is not transfer Krull.

It remains to prove Claim 5.  By \cite[Proposition 3.12]{Ge-Sc16b}, we have
\begin{align*}
		&\mathcal L (C_3^2) = \big\{ \{ 1 \}  \big\} \cup \big\{ [2k, \nu] \mid k \in \N_0 , \nu \in [2k, 5k] \big\} \cup \big\{ [2k+1, \nu] \mid k \in \N, \nu \in [2k+1, 5k+2] \big\}\\
		&=\{0\}\cup \{y+2k+[0,3k], y+2k+1+[0,3k+1], y+2k+2+[0, 3k+2]\mid y,k\in \N_0 \text{ with }y+k\neq 0\}.
		\end{align*}
Let $y,k\in \N_0$. Then $y+2k+[0,3k]=y\{1\}+k[2,5]$, $y+2k+2+[0,3k+2]=y\{1\}+[2,4]+k[2,5]$, and if $k\ge 1$, then $y+2k+1+[0,3k+1]=y\{1\}+[3,7]+(k-1)[2,5]$. Thus $\mathcal L(C_3^2)$ is generated by $\{1\}, [2,3], [2,4], [2,5], [3,7]$ and all these elements are atoms. To verify cancellativity, we use that  all sets of lengths are intervals. Let $a,a',b,b',c,c'\in \N_0$ such that $[a,a'],[b,b'], [c,c']\in \mathcal L(C_3^2)$ and $[a,a']+[c,c']=[b,b']+[c,c']$. Then $a+c=b+c$ and $a'+c'=b'+c'$ which imply that $a=b$ and $a'=b'$. Therefore $[a,a']=[b,b']$ whence $\mathcal L(C_3^2)$ is cancellative.
	
	 Assume to the contrary that there is a transfer homomorphism $\theta\colon \mathcal L(C_3^2)\rightarrow \mathcal B(G_0)$, where $G_0$ is a subset of any abelian group. Since $3[3,7]=\{1\}+4[2,5]$, we have that $\supp(\theta([2,5]))\subset \supp(\theta([3,7]))$ and $\mathsf v_g(\theta([3,7]))\ge \frac{4}{3}\mathsf v_g(\theta([2,5]))>\mathsf v_g(\theta([2,5]))$ for all $g\in \supp(\theta([2,5]))$. Therefore $\theta([2,5])\mid_{\mathcal B(G_0)} \theta([3,7])$, a contradiction as $\theta([3,7])$ and $\theta([2,5])$ are distinct atoms.
\end{proof}

\section{Weakly Krull monoids} \label{7}

In this section we study  weakly Krull monoids and we start with primary monoids. Primary monoids are weakly Krull and  localizations of weakly Krull monoids at minimal nonzero prime ideals are primary.

A monoid $H$ is {\it primary} if it is cancellative with $H \ne H^{\times}$ and for every $a, b \in H \setminus H^{\times}$ there is $n \in \N$ such that $b^n \in aH$. The multiplicative monoid $R^{\bullet}$ of a domain $R$ is primary if and only if $R$ is one-dimensional and local (\cite[Proposition 2.10.7]{Ge-HK06a}). Additive submonoids of $(\Q_{\ge 0}, +)$, called {\it Puiseux monoids}, have found a well-deserved attention in recent literature and  are primary (provided that they are different from $\{0\})$.
Since primary monoids need not be atomic, we restrict to a class of primary monoids (called strongly primary) which are BF-monoids. A monoid $H$ is {\it strongly primary} if it is cancellative with $H \ne H^{\times}$ and for every $a \in H \setminus H^{\times}$ there is $n \in \N$ such that $(H \setminus H^{\times})^n \subset aH$. We denote by $\mathcal M (a)$ the smallest $n \in \N$ having this property. Every primary Mori monoid is strongly primary. Thus numerical monoids are strongly primary and the multiplicative monoids $R^{\bullet}$ of one-dimensional local Mori domains $R$ are strongly primary. An additive submonoid $H \subset (\N_0^s, +)$, with $s \in \N$, is a BF-monoid and it is primary if and only if $H = (H \cap \N^s) \cup \{\boldsymbol 0\}$.

Our first lemma unveils that primary monoids and Krull monoids are very different, both from the algebraic as well as from the arithmetic point of view.

\begin{lemma} \label{7.1}
Let $H$ be a  primary monoid.
\begin{enumerate}
\item $H$ is a Krull monoid if and only if $H = H^{\times} \times H_0$ with $H_0 \cong (\N_0, +)$.

\item If $H$ is strongly primary, then $H$ is a \BF-monoid.

\item Let $H$ be strongly primary. If $H$ is not half-factorial, then  there is a $\beta \in \Q_{> 1}$ such that $\rho (L) \ge \beta$ for all $L \in \mathcal L (H)$ with $\rho (L) \ne 1$. In particular, $H$ is  transfer Krull if and only if it is half-factorial.
\end{enumerate}
\end{lemma}

\begin{proof}
1. Suppose that $H$ is a Krull monoid. Then there is a free abelian monoid $\mathcal F (P)$ such that the inclusion $H_{\red} \hookrightarrow \mathcal F (P)$ is a divisor theory. Since $H$ is primary, it follows that $\supp (aH) = \supp (bH)$ for all $a, b \in H \setminus H^{\times}$. Since every $p \in P$ is a greatest common divisor of elements from $H_{\red}$, it follows that $|P|=1$. Since $H_{\red}$ is completely integrally closed, it is equal to $\mathcal F (P)$. Thus Theorem \ref{4.3}.2 implies that $H$ has the asserted form, and the converse implication is obvious.

2. We assert that every $a \in H \setminus H^{\times}$ has a factorization into atoms and that $\sup \mathsf L (a) \le \mathcal M (a)$. Let $a \in H \setminus H^{\times}$ be given. If $a$ is not an atom, then there are $a_1, a_2 \in H \setminus H^{\times}$ such that $a=a_1a_2$. Proceeding by induction, we obtain a product decomposition of $a$ into $n$ non-units, say $a = a_1 \cdot \ldots \cdot a_n$. If $n > \mathcal M (a)$, then $a_1 \cdot \ldots \cdot a_{n-1} \subset (H \setminus H^{\times})^{\mathcal M (a)} \subset aH$ and hence $a$ divides a proper subproduct of $a_1 \cdot \ldots \cdot a_n = a$, a contradiction. Thus $a$ has a product decomposition into atoms and the number of factors is bounded by $\mathcal M (a)$.

3. The first claim follows from \cite[Theorem 5.5]{Ge-Sc-Zh17b}. Thus Theorem \ref{5.5} and Example \ref{5.4}.2 imply the second statement.
\end{proof}

The arithmetic of various classes of strongly primary monoids, especially of numerical monoids, has found wide attention in the literature. We mention some striking recent results. O'Neill and Pelayo showed that for every finite nonempty subset $C \subset \N_{\ge 2}$ there is a numerical monoid $H$ such that $\Ca (H) = C$ (\cite{ON-Pe18a}). It is an open problem whether there is a numerical monoid $H$ with prescribed sets of distances (see \cite{Co-Ka17a}).
F. Gotti proved that there is a primary BF-submonoid $H$ of $(\Q_{\ge 0}), +)$  such that every finite nonempty set $L \subset \N_{\ge 2}$ occurs as a set of lengths of $H$ (see \cite[Theorem 3.6]{Go19a}, and compare with Remark \ref{5.7}). Such an extreme phenomenon cannot happen if we impose a further finiteness condition, namely local tameness. Let $H$ be a cancellative atomic monoid. For an atom $u \in \mathcal A (H_{\red})$, the {\it local tame degree} $\mathsf t (H,u)$ is the smallest $N \in \N_0 \cup \{\infty\}$ with the following property:
\begin{itemize}
\item[] If $a \in H$ with $\mathsf Z (a) \cap u \mathsf Z (H) \ne \emptyset$, and $z \in \mathsf Z (a)$, then there exists $z' \in \mathsf Z (a) \cap u \mathsf Z (H)$ such that $\mathsf d (z,z') \le N$.
\end{itemize}
The monoid $H$ is {\it locally tame} if $\mathsf t (H, u) < \infty$ for all $u \in \mathcal A (H_{\red})$. If $H$ is finitely generated or a Krull monoid with finite class group, then $H$ is locally tame. Strongly primary monoids with nonempty conductor and all strongly primary domains are locally tame (\cite{Ge-Ro19b}).

Our next result should be compared with Theorem \ref{3.1}, which gathered arithmetical finiteness properties of finitely generated monoids. The main difference is that unions of sets of lengths can be infinite. To give a simple example for this phenomenon, consider the additive monoid $H = \N^2 \cup \{(0,0)\} \subset (\N_0^2, +)$. Then $H$ is a locally tame strongly primary monoid, that is not finitely generated and  $\mathcal U_k (H) = \N_{\ge 2}$ for all $k \ge 2$.

\begin{theorem}[\bf Arithmetic of strongly primary monoids] \label{7.2}
Let $H$ be a locally tame strongly primary monoid.

\begin{enumerate}
\item The set of  catenary degrees and the set of distances   are finite.

\item There is $M \in \N_0$ such that, for all $k \in \N$,  the unions $\mathcal U_k (H)$ are  {\rm AAP}s with difference $\min \Delta (H)$ and bound $M$. If $(H \DP \widehat H) \ne \emptyset$, then all sets $\mathcal U_k (H)$ are finite if and only if $\widehat H$ is a valuation monoid.

\item There is $M \in \N_0$ such that every $L \in \mathcal L (H)$ is a finite  {\rm AAP} with difference $\min \Delta (H)$ and bound $M$.
\end{enumerate}
\end{theorem}

\begin{proof}
1. and 3. follow from \cite[Theorems 3.1.1 and 4.3.6]{Ge-HK06a} and for 2. see \cite{Ge-Go-Tr20}.
\end{proof}

Now we consider the global case. Weakly Krull domains were introduced by Anderson, Anderson, Mott, and Zafrullah (\cite{An-Mo-Za92, An-An-Za92b}). A pure multiplicative description and a divisor theoretic characterization are due to Halter-Koch (\cite{HK95a, HK98}). A monoid $H$ is {\it weakly Krull} if it is cancellative,
\[
H = \bigcap_{\mathfrak p \in \mathfrak X (H)} H_{\mathfrak p} \,, \quad \text{and} \quad \{\mathfrak p \in \mathfrak X (H) \colon a \in \mathfrak p \} \ \text{is finite for all} \ a \in H \,.
\]
Note that $H_{\mathfrak p}$ is a primary monoid for all $\mathfrak p \in \mathfrak X (H)$ and a weakly Krull monoid is Krull if and only if $H_{\mathfrak p}$ is a discrete valuation monoid for all $\mathfrak p \in \mathfrak X (H)$. A domain $R$ is weakly Krull if and only if $R^{\bullet}$ is a weakly Krull monoid. The arithmetic of weakly Krull monoids is studied via transfer homomorphisms to $T$-block monoids (see \cite[Sections 3.4 and 4.5]{Ge-HK06a} for $T$-block monoids and the structure of sets of lengths and \cite{Zh19b, Tr19a} for the structure of their unions). We cannot develop these concepts here whence we restrict to the monoid of their divisorial ideals  whose arithmetic can be  deduced easily from the local case.

For the remainder of this section we study the monoid $\mathcal I_v (H)$ of divisorial ideals of  weakly Krull Mori monoids $H$   and the submonoid $\mathcal I_v^* (H)$ of $v$-invertible divisorial ideals. Clearly, $\mathcal I_v^* (H) \subset \mathcal I_v (H)$ is a divisor-closed submonoid. Every one-dimensional Noetherian domain $R$ (in particular, every Cohen-Kaplansky domain) is a weakly Krull Mori domain and in that case we have $\mathcal I_v^* (R) = \mathcal I^* (R)$. A domain $R$ is called divisorial  (see \cite{Ba00b}) if each nonzero ideal is divisorial (i.e., $\mathcal I_v (R) = \mathcal I (R)$). Note that one-dimensional Noetherian domains need not be divisorial.

\begin{proposition} \label{7.3}
Let $H$ be a weakly Krull Mori monoid.
\begin{enumerate}
\item $\mathcal I_v^* (H)$ is a Mori monoid and it is transfer Krull if and only if it is half-factorial.

\item If \ $\mathcal I_v (H)$ is transfer Krull, then $\mathcal I_v^* (H)$ is half-factorial.

\item If $R$ is an order in a quadratic number field, then $\mathcal I (R)$ is transfer Krull if and only if it is half-factorial.

\end{enumerate}
\end{proposition}

\begin{proof}
1. By \cite[Proposition 5.3]{Ge-Ka-Re15a}, $\mathcal I_v^* (H)$ is a Mori monoid and
\begin{equation} \label{structure-1}
\mathcal I_v^* (H) \cong \coprod_{\mathfrak p \in \mathfrak X (H)} (H_{\mathfrak p})_{\red} \,.
\end{equation}
Suppose that $\mathcal I_v^* (H)$ is transfer Krull. By Lemma \ref{5.2}, the divisor-closed submonoids $(H_{\mathfrak p})_{\red}$, for all $\mathfrak p \in \mathfrak X (H)$,  are transfer Krull and hence they are half-factorial by Lemma \ref{7.1}.3. This implies that $\mathcal I_v^* (H)$ is half-factorial. Since all half-factorial monoids are transfer Krull, the reverse implication is obvious.

2. Since $\mathcal I_v^* (H)$ is a divisor-closed submonoid of $\mathcal I_v (H)$, this follows from 1.

3. Let $R$ be an order in a quadratic number field. Then $H =R^{\bullet}$ is weakly Krull Mori, $R$ is divisorial whence $\mathcal I_v (H) \cong \mathcal I_v (R) = \mathcal I (R)$, and $\mathcal I (R)$ is half-factorial if and only if $\mathcal I^* (R)$ is half-factorial (\cite[Theorem 1.1]{Br-Ge-Re20}. Thus the claim follows from 2.
\end{proof}

Proposition \ref{7.3}.1 shows that $\mathcal I_v^* (H)$ is transfer Krull only in the trivial case when it is half-factorial. This need not be true for the monoid $H$ itself. Indeed, there are weakly Krull Mori monoids $H$ (including orders in number fields) that are transfer Krull but not half-factorial (\cite[Theorem 5.8]{Ge-Ka-Re15a}).

\begin{theorem} \label{7.4}
Let $H$ be a weakly Krull Mori monoid with $\emptyset \ne \mathfrak f = (H \DP \widehat H) \subsetneq H$. We set $\mathcal P^* = \{\mathfrak p \in \mathfrak X (H) \colon \mathfrak p \supset \mathfrak f \}$ and $\mathcal P = \mathfrak X (H) \setminus \mathcal P^*$.
\begin{enumerate}
\item The set of  catenary degrees and the set of distances of $\mathcal I_v^* (H)$  are finite.

\item If there is $M' \in \N$ such that $\rho_{k+1} (H_{\mathfrak p}) - \rho_{k} (H_{\mathfrak p}) \le M'$ for all $k \in \N_0$ and all $\mathfrak p \in \mathcal P^*$, then there are $M, k^*  \in \N_0$ such that, for all $k \ge k^*$,  the unions $\mathcal U_k \bigl( \mathcal I_v^* (H) \bigr)$ are  {\rm AAP}s with difference $\min \Delta \bigl( \mathcal I_v^* (H) \bigr)$ and bound $M$. The sets $\mathcal U_k \bigl( \mathcal I_v^* (H) \bigr)$ are finite for all $k \in \N$ if and only if the map $v$-$\spec (\widehat H) \to v$-$\spec (H)$, defined by $\mathfrak p \mapsto \mathfrak p \cap H$, is bijective.

\item There is $M \in \N_0$ such that every $L \in \mathcal L \bigl( \mathcal I_v^* (H) \bigr)$ is an {\rm AAMP} with difference $d \in \Delta \bigl( \mathcal I_v^* (H) \bigr)$ and bound $M$.
\end{enumerate}
\end{theorem}

\begin{proof}
Since the global conductor $(H \DP \widehat H) \ne \emptyset$, the local conductors $(H_{\mathfrak p} \DP \widehat{H_{\mathfrak p}}) \ne \emptyset$ whence $H_{\mathfrak p}$ is a primary Mori monoid (whence strongly primary) with nonempty conductor (whence locally tame) for all $\mathfrak p \in \mathfrak X (H)$. Since $H$ is a Mori monoid, the set $\mathcal P^*$ is finite and $H_{\mathfrak p}$ is a discrete valuation monoid for all $\mathfrak p \in \mathcal P$. Thus, by \eqref{structure-1}, we infer that
\begin{equation} \label{structure-2}
\mathcal I_v^* (H) \cong \mathcal F (\mathcal P) \times \prod_{\mathfrak p \in \mathcal P^*} (H_{\mathfrak p})_{\red} \,.
\end{equation}

1. This follows from Theorem \ref{7.2}.1 and from the structure of $\mathcal I_v^* (H)$ as given in \eqref{structure-2}.

2. We first note that if $H_1$ and $H_2$ are BF-monoids but not groups and $k \in \N_0$, then
\begin{equation} \label{direct-product}
\mathcal U_k( H_1 \times H_2) = \cup_{\nu \in [0,k]} \bigl( \mathcal U_{k-\nu}(H_1) + \mathcal U_{\nu} (H_2) \bigr) \quad \text{and}
\end{equation}
\begin{equation} \label{direct-product-2}
\rho_k (H_1 \times H_2) = \max \{ \rho_{k-\nu} (H_1) + \rho_{\nu} (H_2) \colon \nu \in [0,k] \} \,.
\end{equation}
Let $M' \in \N$ such that $\rho_{k+1} (H_i) - \rho_k (H_i) \le M'$ for all $k \ge k^*$ and all $i \in [1,2]$. We assert that $\rho_{k+1} (H_1 \times H_2) - \rho_k (H_1 \times H_2) \le M'$ for all $k \ge k^*$.  If $\rho_{k+1} (H_1 \times H_2) = \rho_{k+1} (H_2)$, then $\rho_{k+1} (H_1 \times H_2) - \rho_k (H_1 \times H_2) \le \rho_{k+1}(H_2) - \rho_k (H_2) \le M'$. Otherwise, there is $\nu \in [0,k]$ such that $\rho_{k+1}(H_1 \times H_2) = \rho_{k+1-\nu}(H_1)+\rho_{\nu} (H_2)$ whence $\rho_{k+1}(H_1 \times H_2) - \rho_k (H_1 \times H_2) \le \bigl( \rho_{k+1-\nu}(H_1)+\rho_{\nu} (H_2) \big) - \bigl( \rho_{k-\nu}(H_1) + \rho_{\nu}(H_2)   \bigr) = \rho_{k+1-\nu} (H_1) - \rho_{k-\nu}(H_1) \le M'$.

By assumption, by \eqref{structure-2}, and by the  argument above, we infer that
$\rho_{k+1} \bigl( \mathcal I_v^* (H) \bigr) - \rho_{k} \bigl(  \mathcal I_v^* (H) \bigr) \le M'$ for all $k \ge k^*$. This property and the finiteness of the set of distances imply that the set $\mathcal U_k \bigl( \mathcal I_v^* (H) \bigr)$ have the asserted structure by \cite[Theorem 4.2]{Ga-Ge09b}.

Equation \eqref{direct-product} shows that the sets $\mathcal U_k \bigl( \mathcal I_v^* (H) \bigr)$ are finite for all $k \in \N$ if and only if the sets $\mathcal U_k \bigl( H_{\mathfrak p} \bigr)$ are finite for all $k \in \N$ and all $\mathfrak p \in \mathcal P^*$. Let $\mathfrak p \in \mathcal P^*$. Then $\widehat{H_{\mathfrak p}}$ is Krull and it is a valuation monoid if and only if it is a discrete valuation monoid. Thus, by Theorem \ref{7.2}.2, all sets $\mathcal U_k \bigl( H_{\mathfrak p} \bigr)$ are finite if and only if all $\widehat{H_{\mathfrak p}}$ are discrete valuation monoids if and only if the map $v$-$\spec (\widehat H) \to v$-$\spec (H)$ is bijective.

3. By \eqref{structure-2},
\[
\mathcal L \bigl( \mathcal I_v^* (H) \bigr) = \bigl\{ \{k\} + \sum_{\mathfrak p \in \mathcal P^*} \mathsf L (a_{\mathfrak p}) \colon k \in \N_0, a_{\mathfrak p} \in (H_{\mathfrak p})_{\red} \bigr\} \,.
\]
Thus sets of lengths of $\mathcal I_v^* (H)$ are finite sumsets of sets of lengths of a free abelian monoid and of finitely many locally tame strongly primary monoids. Therefore, by Theorem \ref{7.2}.3, they are sumsets of AAPs and the claim follows by application of an addition theorem given in \cite[Theorem 4.2.16]{Ge-HK06a}. The fact that the difference $d$ lies in $\Delta \bigl( \mathcal I_v^* (H) \bigr)$  can be seen either from a direct argument or one uses \cite[Theorem 4.5.4]{Ge-HK06a}.
\end{proof}

\smallskip
\noindent
{\bf Acknowledgement.} We would like to thank the reviewers for all their helpful comments.

\providecommand{\bysame}{\leavevmode\hbox to3em{\hrulefill}\thinspace}
\providecommand{\MR}{\relax\ifhmode\unskip\space\fi MR }
\providecommand{\MRhref}[2]{%
  \href{http://www.ams.org/mathscinet-getitem?mr=#1}{#2}
}
\providecommand{\href}[2]{#2}

\end{document}